\documentclass[10pt, leqno]{article}
\newenvironment{proof}{\noindent {\bf Proof:}}{$\Box$ \vspace{2 ex}}

\usepackage{amsmath, amssymb, amscd, hyperref}
\usepackage{latexsym, epsfig, color, relsize}
\usepackage[all]{xy}
\usepackage{tikz, verbatim, xcolor, geometry, ntheorem}

\newcommand{\C}{\mathbb{C}}

\newcommand{\ns}{{\rm{ns}}}
\newcommand{\PGL}{{\rm{PGL}}}
\newcommand{\Tr}{{\rm{Tr}}}
\newcommand{\chr}{{\rm{char}}}
\newcommand{\Z}{\mathbb{Z}}
\renewcommand{\P}{\mathbb{P}}

\newcommand{\R}{\mathbb{R}}
\newcommand{\F}{\mathbb{F}}
\newcommand{\A}{\mathbb{A}}
\newcommand{\cO}{\mathcal{O}}
\newcommand{\co}{\mathcal{O}}

\newcommand{\ra}{\rangle}
\newcommand{\la}{\langle}
\newcommand{\beq}{\begin{equation}}
\newcommand{\eeq}{\end{equation}}
\newcommand{\calO}{\mathcal{O}}

\newcommand\Disc{\operatorname{Disc}}
\newcommand\disc{\operatorname{Disc}}

\newcommand\Stab{\operatorname{Stab}}

\newcommand\SL{\operatorname{SL}}
\newcommand\GL{\operatorname{GL}}
\newcommand\gl{\operatorname{GL}}

\newcommand{\rank}{{\rm rank}}

\DeclareMathOperator{\sym}{Sym}

\newcommand{\pb}[6]
{		\left[
			\begin{smallmatrix}%
				#1&#2&#3\\%
				#4&#5&#6\\%
			\end{smallmatrix}%
		\right]%
}

\def\pt#1#2#3#4#5#6#7#8#9{%
	\def\ArgsTenAndFurther##1##2##3{%
		\left[
			\begin{smallmatrix}%
				#1&#2&#3&#4&#5&#6\\%
				#7&#8&#9&##1&##2&##3%
			\end{smallmatrix}%
		\right]%
	}%
	\ArgsTenAndFurther%
}

\newcommand{\twtw}[4]{[\begin{smallmatrix}#1&#2\\#3&#4\\\end{smallmatrix}]}

\newcommand{\odr}{\co_{D1^2}}		% Ramify
\newcommand{\ods}{\co_{D11}}		% Split
\newcommand{\odi}{\co_{D2}}		% Inert
\newcommand{\odg}{\co_{D\rm{ns}}}	% Generic
\newcommand{\ocs}{\co_{C\rm{s}}}	% Singular
\newcommand{\ocg}{\co_{C\rm{ns}}}	% Generic
\newcommand{\ots}{\co_{B11}}		% Split
\newcommand{\oti}{\co_{B2}}		% Inert
\newcommand{\sbs}[2]{W_{[#1,#2]}}
\newcommand{\sby}[2]{Y_{[#1,#2]}}

\newcommand\cF{\mathcal{F}}
\newcommand\Sym{\operatorname{Sym}}

\newtheorem{proposition}{Proposition}%[section]
\newtheorem{theorem}[proposition]{Theorem}
\newtheorem{corollary}[proposition]{Corollary}

\newtheorem{lemma}[proposition]{Lemma}
\newtheorem{remark}[proposition]{Remark}

\newtheorem{assmp}{Assumption}

\newtheorem*{theorem-non}{Theorem}

\definecolor{dgreen}{RGB}{0, 170, 0}

\title{Orbital exponential sums for prehomogeneous vector spaces}
    
\author{Takashi Taniguchi and Frank Thorne}

\begin{document}

\maketitle

\begin{center}
	\large{\itshape{Dedicated to Professor Tomohide Terasoma on his sixtieth birthday}}
\end{center}

\begin{abstract}
Let $(G, V)$ be a prehomogeneous vector space, let $\co$
be any $G(\F_q)$-invariant subset of $V(\F_q)$,
and let $\Phi$ be the characteristic function of $\co$. 
In this paper we develop a method for explicitly and efficiently 
evaluating the Fourier transform $\widehat{\Phi}$,
based on combinatorics and linear algebra.
We then carry out these computations in full
for each of five prehomogeneous vector spaces,
including the $12$-dimensional space of
pairs of ternary quadratic forms.
Our computations reveal that these Fourier transforms enjoy a great deal
of structure, and sometimes exhibit more than square root cancellation on average.

These Fourier transforms naturally arise in analytic number theory, where explicit formulas (or upper bounds)
lead to {\itshape sieve level of distribution} results for related arithmetic sequences. 
We describe some examples, and in the companion paper \cite{TT_leveldist}
we develop a new method to do so, designed to exploit the particular structure of these Fourier transforms.
\end{abstract}

\section{Introduction}
Our results are best illustrated by example. Let $V = \Sym^3(2)$ be the space of binary cubic forms,
together with an action of $G := \GL_2$ given by
\beq
(g \circ x)(u, v) = \frac{1}{\det(g)} x((u, v) g).
\eeq
where $x=x(u,v)\in V$ is a binary cubic form.

The pair $(G,V)$ is {\itshape prehomogeneous}: over an algebraically closed field $k$, the action of $G(k)$ on $V(k)$ has
a Zariski open orbit -- here the locus of points $x \in V$ with $\Disc(x) \neq 0$. Any binary cubic form
is determined up to a scalar multiple by its roots in $\P^1(k)$, and prehomogeneity
is equivalent to the fact that $\PGL_2$ acts triply transitively on $\P^1$.

If $k$ is any field with $\chr(k) \neq 3$, then (as we will explain later) there is 
a bilinear form $[-,-] : V(k) \times V(k) \rightarrow k$ for which
$[g x, g^{-T} y] = [x, y]$ for all
$x, y \in V(k)$ and $g \in G(k)$.
Here $g^{-T}\in G(k)$ is the inverse of the transpose of $g$.
If further $k = \F_q$ is a {\itshape finite} field of characteristic $p$ and 
$\Phi: V(\F_q) \rightarrow \C$ is any function,
we define its Fourier transform $\widehat{\Phi} : V(\F_q) \rightarrow \C$ by
\beq\label{eqn:ft_intro}
\widehat{\Phi}(y) = q^{-4} \sum_{x \in V(\F_q)} \Phi(x) \exp\left(2 \pi \sqrt{-1}  \cdot \frac{\Tr_{\F_q/\F_p}([x, y])}{p}\right).
\eeq

The following elementary result is the prototype for the kind of result we are after:
\begin{proposition}\label{thm:Phip}
Let $w_q : V(\F_q) \rightarrow \C$ be the counting function of the number of roots of $x \in V(\F_q)$
in $\P^1(\F_q)$. Then, assuming that $\chr(\F_q) \neq 3$, we have
\begin{equation}\label{eq:Psi_fourier}
\widehat{w_q}(x)=
\begin{cases}
	1+q^{-1}	& x=0,\\
	q^{-1}		& \text{$x \neq 0$ and $x$ has a triple root in $\P^1(\F_q)$},\\
	0		& \text{otherwise}.
\end{cases}
\end{equation}
\end{proposition}
Work of Davenport and Heilbronn \cite{DH} connected this $(G,V)$ to counting problems involving
cubic fields and $3$-torsion in the class groups of quadratic fields, 
and 
formulas of a similar shape to \eqref{eq:Psi_fourier} appeared in subsequent works including the following:
\begin{enumerate}
\item
The function $w_p(x)$ appears in Bhargava, Shankar, and Tsimerman's
\cite[(80)-(83)]{BST} proof of negative secondary terms
in the Davenport-Heilbronn theorem.
These authors proved and applied the weaker result
\cite[(80)-(83)]{BST} that $|\widehat{w_p}(x)| \ll p^{-1}$ for all $x \neq 0$.
\item
Related functions, defined over $V(\Z/p^2\Z)$, appear in the authors' \cite{TT_L, TT_rc} (independent) proof of these same secondary terms.
Coarse bounds did not suffice for our methods, and
we obtained exact formulas, but by rather laborious methods.
\item
As we will shortly describe, 
Belabas and Fouvry proved \cite[Corollaire 2]{BF} that there are
infinitely many cubic fields whose discriminant is fundamental and divisible by at most $7$ prime factors. They relied
on similar exponential sum estimates which are proved in \cite[Section 3]{BF}.

\end{enumerate}
Our aim was to develop a simple and generalizable method for proving exact formulas of this shape. In this paper,
we will describe our method and compute
the Fourier transforms of the characteristic functions of each
of the $G(\F_q)$-orbits on each of the following prehomogeneous vector spaces:
\begin{itemize}
\item
$V = \Sym^3(2)$, the space of binary cubic forms; $G = \GL_2$.
\item
$V = \Sym^2(2)$, the space of binary quadratic forms; $G = \GL_1 \times \GL_2$.
\item
$V = \Sym^2(3)$, the space of ternary quadratic forms; $G = \GL_1 \times \GL_3$.
\item
$V = 2\otimes \Sym^2(2)$, the space of pairs of binary quadratic forms; $G = \GL_2 \times \GL_2$.
\item
$V = 2\otimes \Sym^2(3)$, the space of pairs of ternary quadratic forms; $G = \GL_2 \times \GL_3$.
\end{itemize}
We thus obtain Fourier transform formulas for any $G$-invariant function $\Phi$, i.e. one
which satisfies $\Phi(g x) = \Phi(x)$ for all $g \in G(\F_q)$ and $x \in V(\F_q)$.
We exclude finitely many field characteristics in each case,
but otherwise our results are completely general.
Our formulas for $(G,V)$ above are respectively
given in Theorems \ref{thm:sym32-A}, \ref{thm:sym22-A},
\ref{thm:sym23-A}, \ref{thm:2sym22-A} and \ref{thm:2sym23-A}.
We have worked out some additional cases as well,
but we leave the details for subsequent papers.

A sample result (in addition to the simpler Proposition \ref{thm:Phip}) is as follows:

\begin{theorem}\label{thm:ternary}
For a finite field $\F_q$ of characteristic not equal to $2$,
let $V(\F_q) := \F_q^2 \otimes \Sym^2(\F_q^3)$ be the space of pairs of ternary quadratic forms, and
write $\Psi_q\colon V(\F_q)\rightarrow\{0,1\}$ for
the characteristic function of those $x \in V(\F_q)$ which are singular.

Then, we have
\begin{equation}\label{eq:Psi_fourier_ternary}
\widehat{\Psi_q}(x)=
\begin{cases}
q^{-1} + 2q^{-2} - q^{-3} - 2 q^{-4} - q^{-5} + 2q^{-6} + q^{-7} - q^{-8}
	& x \in \co_0,\\
q^{-3} - q^{-4} - 2q^{-5} + 2q^{-6} + q^{-7} - q^{-8}
	& x \in \odr,\\
2q^{-4} - 5q^{-5} + 3q^{-6} + q^{-7} - q^{-8}
	& x \in \ods,\\
q^{-4} - 3q^{-5} + 2q^{-6} + q^{-7} - q^{-8}
	& x \in \ocs,\\
- q^{-5} + q^{-6} + q^{-7} - q^{-8}
	& x \in \odi, \odg, \ocg, \ots, \oti,\\
-q^{-6} + 2q^{-7} - q^{-8}
	& x \in \cO_{1^2 1^2}, \\
q^{-6} - q^{-8}
	& x \in \cO_{2^2}, \\
q^{-7} - q^{-8}
	& x \in \co_{1^4}, \co_{1^3 1}, \co_{1^2 11}, \co_{1^2 2}, \\
-q^{-8}
	& x \in \cO_{1111}, \cO_{112}, \cO_{22}, \cO_{13}, \cO_{4}.
\end{cases}
\end{equation}
\end{theorem}

The sets on the right are defined, and their cardinalities computed, in Proposition \ref{prop:orbit-2sym23}.
An element $x \in V(\F_q)$ is {\itshape singular} if it belongs to any of the orbits listed before the last line; 
see the introduction to Section \ref{sec:pairs_ternary} for a more intrinsic definition.  

We can see by comparing the above result to Proposition \ref{prop:orbit-2sym23}
that the sizes of $\widehat{\Psi_q}(x)$ and the orbits $\cO$ containing $x$ are approximately inversely correlated.
%For example those $\cO$ with a $D$ in their subscript consist of pairs $(A, B)$ of doubled forms 
%satisfying $\lambda A + \mu B = 0$ for some $\lambda, \mu \in \F_q$, and there are only $O(q^7)$ such pairs.\footnote{Does this explanation enough to say as an example of ``inversely correlated''?}}
%
In particular, on average, we obtain better than square root cancellation:
\begin{corollary}\label{cor:l1} We have the $L_1$-norm bound
\[
\sum_{x \in V(\F_q)} |\widehat{\Psi_q}(x)| \ll q^4.
\]
\end{corollary}

\medskip
We mention two other papers to which our results are related: 

\begin{enumerate}

\item
In an important paper \cite{FK}, Fouvry and Katz obtained {\itshape upper bounds} 
for related exponential sums in a much more
general context. As a special case, let $Y$ be a (locally closed) subscheme 
of $\A_{\Z}^n$, and consider the exponential sum
\eqref{eqn:ft_intro} with $V = \A_{\Z}$, $q = p$ prime, and 
$\Phi$ the characteristic function of $Y(\F_p)$.
Fouvry and Katz
produce a filtration of
subschemes $\A_\Z^n \supseteq X_1 \supseteq \cdots \supseteq X_j \supseteq \cdots \supseteq X_n$ of 
increasing codimension, so that successively weaker upper bounds hold on each
$(\A_\Z^n - X_j)(\F_p)$.
Our Proposition \ref{thm:Phip} and Theorem \ref{thm:ternary} illustrate a similar structure,
with substantially smaller values than the general bounds proved by Fouvry-Katz.

As an interesting application (\cite[Corollary 1.3]{FK}), they prove that there are infinitely many
primes $p \equiv 1 \pmod 4$ for which $p + 4$ is squarefree and not the discriminant of a cubic field.

\item
Denef and Gyoja \cite{DG} studied the sum \eqref{eqn:ft_intro},
in the not necessarily $G$-invariant case defined by
$\Phi(x) = \chi(\Disc(x))$, where $\chi$ is a nontrivial
Dirichlet character modulo $p$, so that $\Phi$ is relatively invariant and supported on the nonsingular orbits.
In this setting, Denef and Gyoja proved
that the Fourier transform of $\Phi(x)$ is equal to
$\chi^{-1}(\Disc(x))$
times a factor
independent of $x$.
Their result has a shape reminiscent to that of Sato's \cite{sato} fundamental
theorem of prehomogeneous vector spaces (over $\C$).

In some cases where $\chi$ is of small order their result excludes singular $v$,
depending on the Sato-Bernstein polynomial of $(G, V)$. When the quadratic character
does define a $G$-invariant function we can recover these cases of their results. The case $\Sym^2(2)$ is particularly interesting, as Denef-Gyoja
does not imply that $\widehat{\Phi}(x) = 0$ for singular $x$.
Indeed, as they discuss in \cite[Remark 5.2.3.3]{DG}, 
Parseval's formula determines whether or not $\widehat{\Phi}$ vanishes on the singular set, and 
in Remark \ref{rem:sym22} we give an alternate proof that it does not. Moreover, $2 \otimes \Sym^2(2)$ provides another such example.

\end{enumerate}

Both of these papers are quite long and invoke the machinery
of sheaf cohomology. Our methods are much simpler,
as we will demonstrate by now giving a complete 
proof of Proposition \ref{thm:Phip}. It would be interesting
to further investigate our results from a cohomological point of view; 
for example, is there a simple reason why the expressions
in \eqref{eq:Psi_fourier_ternary} and Theorem \ref{thm:2sym23-A}
 are fixed polynomials in $q$, while those in Theorem \ref{thm:sym32-A}
depend on $q \pmod{3}$? And are there simple cohomological interpretations
for these polynomials? We leave such questions for further work.
 
 \medskip
To prove  Proposition \ref{thm:Phip}, 
write $\langle n \rangle := \exp(2 \pi\sqrt{-1}\cdot \Tr_{\F_q/\F_p}(n) / p)$, and write
$\Phi_q$ for the characteristic function of the orbit $(1^3)$: those nonzero elements of $V(\F_q)$ which have a triple root.
By Fourier inversion, it suffices to compute the Fourier transform of the right side of \eqref{eq:Psi_fourier}, and thus to 
compute $\widehat{\Phi_q}$.

Using the facts that $(1^3)$ is a single $\GL_2(\F_q)$-orbit,
and that our bilinear form (defined by \eqref{eq:sym32-bilinear-form})
satisfies the $\SL_2(\F_q)$-invariance property $[gx,g^{-T}y]=[x,y]$
(where $g^{-T}$ is the inverse transpose of $g$),
we compute that
\begin{align*}
q^4 \widehat{\Phi_q}(y) = 
\frac{1}{q^2 - q} \sum_{g \in \SL_2(\F_q)} \sum_{t \in \F_q^{\times}} \la [ g \cdot (t, 0, 0, 0), y ] \ra
= & \frac{1}{q^2 - q} \sum_{g \in \SL_2(\F_q)} \sum_{t \in \F_q^{\times}} \la [ (t, 0, 0, 0), g^{T} y ] \ra
\end{align*}
The inner sum is equal to $q - 1$ if $[1 : 0] \in\P^1(\F_q)$ is a root of $g^Ty$, and $-1$ if it is not.
For each root $\alpha$ of $y$, counted with multiplicity, $[1 : 0]$ will be a root of $g^T y$ for $\frac{|\SL_2(\F_q)|}{q + 1} = q^2 - q$ elements $g
\in \SL_2(\F_q)$, so that
\begin{align*}
q^4 \widehat{\Phi_q}(y) = 
& \frac{1}{q^2 - q} \cdot (q^2 - q) \cdot \bigg(q w_q(y)  - (q + 1) \bigg).
\end{align*} 
Proposition \ref{thm:Phip} now follows easily.

Similar ideas can easily be applied to compute characteristic functions of other orbits. Generally, the idea
is to consider subspaces $W \subseteq V$ defined by the vanishing of some of the coordinates (whose orthogonal complements
are also so defined); as in the above example, the number of elements in $W \cap \calO$ and $W^{\perp} \cap \calO$ 
for various $G$-orbits $\calO$ can be computed via elementary geometric considerations, and these counts determine the Fourier transforms.
This is the basic principle which we will develop and apply to obtain all of our Fourier transform formulas.

\medskip

{\bf What is required for the method to work?} Here we offer some thoughts as to what conditions one must demand of a space 
$(G, V)$ to make this method work. 
These comments may be considered somewhat preliminary, as we hope to extend the scope of our method
in subsequent work. 
\begin{itemize}
\item
The $G$-orbits must be scale-invariant.
(In some cases we incorporate a $\GL_1$ factor
into the action to ensure this.)
\item
We assume here that the $G(\F_q)$-orbits on $V(\F_q)$ can be
described uniformly in $q$, and that their number does not depend on $q$.
This is not an absolute requirement; indeed, one may obtain
results for the space ${\rm M}_2(\F_q)$ of $2 \times 2$ matrices, together
with the action of $\GL_2(\F_q)$ by left multiplication
(which has $\gg q$ orbits) with the method described in this paper.
\item
Most significantly, there must be `many' subspaces $W$ for which
we can count the intersections of $W$ and $W^\perp$ with the $G$-orbits,
and we must obtain a suitably large variety of different counting functions.
(See Proposition \ref{prop:A}.) We verified this for each of the five
$(G, V)$ studied here by checking cases; we are not yet aware of any
{\itshape a priori} reason why this should necessarily be true.
\end{itemize}
In a monumental work \cite{saki}, Sato and Kimura classified all irreducible ``reduced'' prehomogeneous vector spaces. There are $36$ of them, including a few infinite families. 
It seems that our method is applicable to (at least) most of those cases to get exact results. Ishitsuka, Ishimoto and the first author are confirming this for more than $10$ more cases, 
including some infinite families such as spaces of $n\times m$ matrices (for any $n$ and $m$) and of alternating forms of general degree (see e.g., \cite{ishimoto-cubic,ishimoto-quadratic,ishi-tani}). We hope that this topic will be pursued further in the near future.

\medskip
{\bf Sieve Applications.} Typically (and in the papers described above), exponential sum bounds lead to {\itshape level of distribution estimates},
which in turn lead to sieve applications. Typically a sieve involves the following:
\begin{itemize}
\item
A set of objects being sieved. For example, with any of the $(G, V)$ studied in this paper, one might consider
the set
of $G(\Z)$-equivalence classes of $x \in V(\Z)$ with $0 < |\Disc(x)| < X$.
\item
For each prime $p$, a notion of an object being `bad at $p$'. Often this means simply that $p \mid \Disc(x)$.
In works \cite{BBP, B_quartic, B_quintic, BST, DH, ST, TT_rc}
on counting number fields, `bad at $p$' is taken to mean that the ring $R$ corresponding to $x$ is nonmaximal at $p$,
i.e. that $R \otimes_{\Z} \Z_p$ is nonmaximal as a cubic ring over $\Z_p$.
\end{itemize}

A typical aim of sieve methods is to fix a large set of primes $\mathcal{P}$ and estimate the number of objects $x$
which are not bad at any $p \in \mathcal{P}$. To carry this out one generally needs,
for squarefree integers $q$, estimates for the number of $x$ bad at each prime divisor of $q$.
Loosely speaking, we say that our sieve 
has {\itshape level of distribution $\alpha > 0$} if we can usefully bound the sum over $q < X^{\alpha}$ of the resulting error terms.

Positive levels of distribution for the nonmaximal definition of `bad' have led
to power saving error terms in certain counting functions for maximal orders, and thus also for the number fields
containing them. Sieving for divisibility leads instead to estimates for almost-prime discriminants of number fields, and 
in our companion paper \cite{TT_leveldist} we obtain such an application:
\begin{theorem}\label{thm:ap}
There is an absolute constant $C > 0$ such that for each $X > 0$,
there exist $\geq (C + o_X(1)) \frac{X}{\log X}$ quartic fields $K$
whose discriminant is fundamental, bounded above by $X$,
and has at most $8$ prime factors.
\end{theorem}

Our methods are closely related to those of Belabas and Fouvry \cite{BF}, who proved \cite[Corollaire 2]{BF} the same for {\itshape cubic} fields.
(They formulated their result in terms of the $2$-torsion in the class group of the quadratic resolvent; see (1.1) of their paper.) They pointed out
that a weighted sieve would reduce their $7$ to $4$,
and we will separately recover their result and
further reduce this $4$ to $3$.

The basic heuristic of \cite{TT_leveldist} is that $L_1$-norm bounds such as
Corollary \ref{cor:l1},
trivially obtained as consequences of the results of this paper,
should lead to corresponding levels of distribution, 
and in \cite{TT_leveldist} we develop a geometric method which approaches this heuristic.
Most of our work in \cite{TT_leveldist} is carried out in a general setting, adaptable to other representations $(G, V)$
and to other sieve applications.

\medskip
{\bf Zeta functions.} Another motivation for our work is that the exponential sums being studied 
arise as coefficients of the functional equations of the associated Sato-Shintani zeta functions. These zeta functions can be used to
prove sieve estimates (see e.g. \cite{TT_L, TT_rc}),
and are also of intrinsic interest
-- especially when it can then be proved that the functional equations
assume a particularly nice form. In Corollary \ref{cor:ohno-nakagawa} (of this paper) we present an application of this type.

\medskip

\begin{remark}
A related computation was subsequently carried out by Hough \cite{hough_quartic_expo}.
Let $V = V(\Z/p^2 \Z) = 2 \otimes \Sym^2(3)$ be the ``space'' of pairs of ternary quadratic forms 
with coefficients in $\Z/p^2\Z$, and let $\Psi_{p^2} \ : V(\Z/p^2\Z) \rightarrow \{ 0, 1 \}$
be the characteristic function of those which are ``nonmaximal at $p$''. Then, for $p > 3$, Hough computes the 
Fourier transform of $\Psi_{p^2}$. 

Hough's computation should yield improved results concerning the distribution of quartic fields, for example to his 
results  \cite{hough_quartic_shape} obtaining a quantitative equidistribution result for their shape.

\end{remark}

{\bf Organization of the paper.} 
We begin in Section \ref{sec:setup} by giving the necessary background and assumptions. Our method
applies to any $\F_q$-linear representation $V$ of a finite group $G$, for which there exists a symmetric bilinear form which behaves
nicely (see Assumption \ref{assmp:dual-identify}) with respect to the $G$-action.  {\itshape There is no assumption that $(G, V)$ 
is prehomogeneous}, although our method was designed to exploit features
typical for prehomogeneous vector spaces.

We define (see \eqref{eq:A}) a matrix $M$ which carries all the necessary information concerning our Fourier transforms. We then prove
Proposition \ref{prop:A}, our main technical input, and explain how it 
reduces the problem of determining $M$ to the combinatorial task of counting $W \cap \calO_i$ for all $G(\F_q)$-orbits $\calO_i$ on $V$,
for a large enough number of subspaces $W$.

In the next five sections we treat the prehomogeneous representations $\Sym^3(2)$, $\Sym^2(2)$, $\Sym^2(3)$, 
$2 \otimes \Sym^2(2)$, and $2 \otimes \Sym^2(3)$ in turn. We describe each representation, determine the $G(\F_q)$-orbits
on $V(\F_q)$ (in each case excluding a `bad' characteristic), carry out the combinatorial problem described above, and determine the
matrix $M$. In the latter three cases embeddings of the previously considered representations will be relevant, 
so that these sections are not independent of the previous ones.
We assume $\chr(\F_q)\neq3$ for $\Sym^3(2)$ and
$\chr(\F_q)\neq2$ for the latter four representations.

Finally, in Appendix \ref{appendix:ibf} we explain why bilinear forms satisfying Assumption \ref{assmp:dual-identify} exist
in a general setting which includes each of the five $(G, V)$ treated here. This seems to be `well known', 
but the constructions are usually presented without proof 
in the related literature and we hope that a complete presentation will be useful to the reader.

\medskip
{\bf Notation.}
For the convenience of the reader we describe some commonly used notation.
We work with a finite field $\F_q$ of characteristic $p$,
and $V$ will always denote a finite dimensional $\F_q$-linear representation
of a finite group $G$. The $G$-orbits on $\F_q$ will be labeled either
$\co_1$, $\co_2$, etc. when we emphasize their ordering with respect
to the matrix $M$, or using descriptive labels such as
$\co_{D1^2}$ and $\co_{1111}$ when we emphasize their arithmetic properties.
These labels will be introduced separately in each section.

If $g\in\gl_n$, then
$g^{-T}\in\gl_n$ is the inverse of the transpose of $g$, and
if $g=(g_1,\dots,g_r)\in\gl_{n_1}\times\dots\times\gl_{n_r}$,
we write $g^{-T}=(g_1^{-T},\dots,g_r^{-T})$.

Some additional notation used in our orbit counting (e.g., $W_{[i, j]}$, $W_{[i, j]}^{\times}$, $Y$)
is introduced and explained in Section \ref{sec:2sym22}.

\section{The basic setup}\label{sec:setup}
In this section we formalize our method.
We start by describing the common features of our
representation over $\F_q$
which are necessary to make the method work;
we then formulate several basic results which we apply
in the course of our proofs.
Specifically, our aim in this section is to
establish Proposition \ref{prop:A},
which is the primary tool for our exponential sum computations.

Let $V$ be a (finite dimensional) vector space over $\F_q$.
Let $V^\ast$
be the dual space,
i.e., the set of linear forms on $V$.
For $x\in V$ and $y\in V^\ast$,
we write $[x,y]:=y(x)\in\F_q$ for the natural pairing 
between $V$ and $V^\ast$.
For $x\in V$ and $y\in V^\ast$, let
\begin{equation}
\langle x,y\rangle:=
\exp\left(2\pi\sqrt{-1}\cdot \frac{{\rm Tr}_{\F_q/\F_p}([x,y])}{p}\right)
\in\C_1^\times.
\end{equation}
Here
$\C_1^\times=\{z\in\C^\times\mid |z|=1\}$.
Then $V\ni x\mapsto\langle x,y\rangle\in\C_1^\times$
is a group homomorphism.
The basic underlying principle for the present
(finite) Fourier analysis is that $V^\ast$
is canonically identified with
the group of additive characters
on $V$, via
$
V^\ast\ni
y\mapsto
\langle\cdot,y\rangle
\in{\rm Hom}(V,\C_1^\times).
$
Let $\cF_V$ and $\cF_{V^\ast}$ be the space of
$\C$-valued functions on $V$ and $V^\ast$, respectively.
There are special $\C$-linear isomorphisms between them;
the Fourier transforms
%$\cF_V\ni \phi\mapsto\phi^\wedge\in\cF_{V^\ast}$
%and
%$\cF_{V^\ast}\ni \psi\mapsto\psi^\wedge\in\cF_V$,
%defined respectively by
\[
\cF_V\ni \phi\longmapsto\widehat\phi\in\cF_{V^\ast};
\qquad
\widehat{\phi}(y)=|V|^{-1}\sum_{x\in V}\phi(x)
\langle x,y\rangle
%	\exp\left(2\pi\sqrt{-1}\cdot \frac{\tr_{\F_q/\F_p}([x,y])}{p}\right)
\]
and
\[
\cF_{V^\ast}\ni \psi\longmapsto\widehat\psi\in\cF_V;
\qquad
\widehat\psi(x)=|V|^{-1}\sum_{y\in V^\ast}\psi(y)
\langle x,y\rangle.
%	\exp\left(-2\pi\sqrt{-1}\cdot \frac{\tr_{\F_q/\F_p}([x,y])}{p}\right).
\]
By the orthogonality relation
we have $\widehat{\widehat\phi\,}(x)=|V|^{-1}\phi(-x)$,
which is the Fourier inversion formula in this case.
For a subspace $W$ of $V$,
let $W^\bot\subset V^\ast$ be the subspace of its annihilators
in the dual space.
Let $\phi_W$ and $\psi_{W^\bot}$ respectively be their indicator functions.
It is easy to see that
\begin{equation}\label{eq:subspace-FT}
\widehat{\phi_W}=\frac{|W|}{|V|}\cdot \psi_{W^\bot}.
\end{equation}

If a finite group $G$ acts on $V$,
then the action is naturally inherited by $\cF_V$:
for $g\in G$ and $\phi\in\cF_V$, defining
$g\phi\in\cF_V$ by $(g\phi)(x)=\phi(g^{-1}x)$
defines a $\C$-linear representation of $G$.
Let $\cF_V^G\subset \cF_V$
be the subspace of $G$-invariant functions on $V$.
If $\co_1,\dots,\co_r$ are all the distinct $G$-orbits in $V$,
then their respective indicator
functions $e_1,\dots,e_r\in \cF_V^G$
form a basis of $\cF_V^G$.
As is common, the averaging operator
\[
{\rm av}\colon \cF_V\longrightarrow\cF_V^G;
\qquad
\phi\longmapsto
{\rm av}(\phi):=
|G|^{-1}\sum_{g\in G}g\phi
\]
is useful for our analysis.
If $\phi_x$ is the indicator function of a point $x\in V$,
then
$
{\rm av}(\phi_x)
=e_i/|\co_i|,
$
where $x\in\co_i$.
Hence if $\phi_X$
is the indicator function of a subset $X\subset V$,
then since
$\phi_X=\sum_{x\in X}\phi_x$
we have
\begin{equation}\label{eq:av-indicator-fn}
{\rm av}(\phi_X)
=\sum_{1\leq i\leq r}\frac{|\co_i\cap X|}{|\co_i|}\cdot e_i.
\end{equation}

We now assume that the action of $G$ on $V$
is $\F_q$-linear.
Given any automorphism $G\ni g \mapsto g^\iota\in G$
of order $1$ or $2$ (as will be discussed shortly), we
consider the action
of $G$ on $V^\ast$ defined by
$[x,gy]=[(g^\iota)^{-1}x,y]$.
It is easy to see that this action is well-defined and
$\F_q$-linear, and that the action thus defined on $V^{\ast\ast}$ is equivariant with respect to the canonical isomorphism $V \rightarrow V^{\ast\ast}$.
We thus have a $\C$-linear representation of $G$ on $\cF_{V^\ast}$,
the subspace of $G$-invariant functions
$\cF_{V^\ast}^G$, and the averaging operator
${\rm av}\colon \cF_{V^\ast}\rightarrow\cF_{V^\ast}^G$
as well.

For the Fourier transform,
we immediately see that
\begin{equation}\label{eq:G-eq}
\widehat{g\phi}=g^\iota\widehat\phi,
\qquad
g\in G, \phi\in\cF_V.
\end{equation}
In particular, if $\phi$ is $G$-invariant,
then so is $\widehat\phi$. Thus we have
a $\C$-linear isomorphism
\[
\cF_V^G\longrightarrow \cF_{V^\ast}^G;
\qquad
\phi\longmapsto\widehat\phi.
\]
Our goal is to understand this
Fourier transform between
$\cF_V^G$
and
$\cF_{V^\ast}^G$
explicitly.
By looking at their dimensions, we see that
$V$ and $V^\ast$ have the same number of $G$-orbits.
Let
$\co_1^\ast,\dots,\co_r^\ast$ be the all $G$-orbits in $V^\ast$,
and $e_i^\ast\in\cF_{V^\ast}^G$
be the indicator function of $\co_i^\ast$.
Let $M\in{\rm M}_r(\C)$
be the representation matrix
of the Fourier transformation
with respect to the basis $e_1,\dots,e_r\in\cF_V^G$
and $e_1^\ast,\dots,e_r^\ast\in\cF_{V^\ast}^G$.
By definition,
\begin{equation}\label{eq:A}
M=[a_{ij}]
\quad
\text{where}
\quad
\widehat{e_j}=\sum_{i}a_{ij}e_i^\ast,
\end{equation}
and we wish to determine this matrix $M$.

By \eqref{eq:G-eq}, we have
$
\widehat{{\rm av}(\phi)}
={\rm av}(\widehat\phi).
$
Let $W$ be any subspace of $V$ and put $\phi=\phi_W$.
Then by \eqref{eq:subspace-FT} and \eqref{eq:av-indicator-fn},
we have the following simple formula, which is particularly useful:
\begin{equation}\label{eq:subspace-fourier-average-general}
\sum_{1\leq i\leq r}\frac{|\co_i\cap W|}{|\co_i|}\cdot\widehat{e_i}
=\frac{|W|}{|V|}
\sum_{1\leq i\leq r}\frac{|\co_i^\ast\cap W^\perp|}{|\co_i^\ast|}\cdot e_i^\ast.\end{equation}

We now consider the following assumption on
the representation $(G,V)$, which is satisfied
by all of the cases we will study in later sections.
(We will be required to assume that the characteristic of $\F_q$
is not one of finitely many `bad' primes, depending on the particular
$(G, V)$.)
\begin{assmp}\label{assmp:dual-identify}
There is an involution
$G\ni g\mapsto g^\iota\in G$ and a non-degenerate
symmetric bilinear form
$V\times V\ni(x,y)\rightarrow b(x,y)\in\F_q$,
such that $b(gx,g^\iota y)=b(x,y)$.
\end{assmp}

By non-degeneracy
the map $\theta\colon V\ni y\mapsto b(\cdot, y)\in V^\ast$
is an $\F_q$-linear isomorphism.
We consider the representation of $V^\ast$
with respect to the involution $\iota$. Then $\theta$
preserves the action of $G$.
We identify $V^\ast$ with $V$ via $\theta$.
Hence the pairing on $V$ and $V^\ast=V$ is given by
$[x,y]=b(x,y)$,
and this satisfies $[gx,g^\iota y]=[x,y]$.
The Fourier transform is now an automorphism
on $\cF_V$ or on $\cF_V^G$.
We put $\co_i^\ast=\co_i$ and thus $e_i^\ast=e_i$.
We summarize our argument
into the following.
% \eqref{eq:subspace-fourier-average-general}, we have:
\begin{proposition}\label{prop:A}
Let $V$ be a finite dimensional representation of a group $G$
over $\F_q$,
satisfying Assumption \ref{assmp:dual-identify}.
Then for any subspace $W$ of $V$, we have
\begin{equation}\label{eq:subspace-fourier-average-selfdual}
\sum_{1\leq i\leq r}\frac{|\co_i\cap W|}{|\co_i|}\cdot\widehat{e_i}
=\frac{|W|}{|V|}
\sum_{1\leq i\leq r}\frac{|\co_i\cap W^\perp|}{|\co_i|}\cdot e_i.
\end{equation}
\end{proposition}

We can now precisely explain our strategy for determining $M$: given $(G, V)$,
we compute the vectors $(|\co_i\cap W|)_i$ and $(|\co_i\cap W^\perp|)_i$ for many subspaces $W$.
Eventually these vectors will span $\R^r$, after which basic linear
algebra finishes off the computation.

We have not attempted to prove in general that the vectors $(|\co_i\cap W|)_i$ span $\R^r$, but in 
practice this does not seem to be an issue.

\medskip

Before ending this section,
we prove an additional lemma on the matrix $M$.
This lemma is logically not necessary,
but we find it quite convenient to check our computation.

\begin{lemma}\label{lem:A}
\begin{enumerate}
\item
We have $|\co_i|a_{ij}=|\co_j|a_{ji}$.
Namely, if we put $S={\rm diag}(|\co_i|)$,
then $SM$ is symmetric.
\item
Suppose that $x$ and $-x$ lie in the same $G$-orbit
for each $x\in V$. Then $M^2=|V|^{-1}E_r$,
where $E_r$ is the identity matrix.
\end{enumerate}
\end{lemma}
\begin{proof}
For (1), first note that
$
\widehat{\phi_x}(y)
=|V|^{-1}\langle x,y\rangle
=\widehat{\phi_y}(x)
$
for all $x,y\in V$. Hence we also have
\[
\widehat{{\rm av}(\phi_{x})}(y)
=\frac{1}{|G|}\sum_{g\in G}\widehat{g\phi_{x}}(y)
=\frac{1}{|G|}\sum_{g\in G}\widehat{\phi_{gx}}(y)
=\frac{1}{|G|}\sum_{g\in G}\widehat{\phi_{y}}(gx)
={\rm av}(\widehat{\phi_{y}})(x)
=\widehat{{\rm av}(\phi_{y})}(x).
\]
If $x\in\co_j$ and $y\in\co_i$,
then since ${\rm av}(\phi_x)=|\co_j|^{-1}e_j$ we have
\[
\widehat{{\rm av}(\phi_{x})}(y)
=|\co_j|^{-1}\widehat{e_j}(y)
=|\co_j|^{-1}\sum_ka_{kj}e_k(y)
=|\co_j|^{-1}a_{ij},
\]
and similarly 
$\widehat{{\rm av}(\phi_{y})}(x)
=|\co_i|^{-1}a_{ji}$.
For (2), by assumption
$\phi(-x)=\phi(x)$ if $\phi\in\cF_V^G$.
Thus $M^2=|V|^{-1}E_r$ is simply the Fourier inversion.
\end{proof}

In successive sections, we obtain
$M$ in Theorems
\ref{thm:sym32-A},
\ref{thm:sym22-A},
\ref{thm:sym23-A},
\ref{thm:2sym22-A}
and
\ref{thm:2sym23-A}
for each of the cases.
We double checked our computation
by confirming that $M$'s in the theorems all satisfy
Lemma \ref{lem:A}.

We used PARI/GP \cite{pari} to carry out the necessary linear algebra. In each case 
we have embedded our source code, together with the matrix $M$ in machine-readable format,
as a comment immediately following the theorem statement in the \LaTeX \ source for this file, which may be freely downloaded from the arXiv.
The source code is also available on the second author's website. %(\href{http://people.math.sc.edu/thornef}{\underline{link}}).

\section{$\Sym^3(2)$}\label{sec:sym32}
We first handle the space of binary cubic forms, as we described in the introduction. In this case the matrix $M$ was determined
previously by Mori \cite{mori}, and so here we give a second proof.

Let $V = \Sym^3(\F_q^2)$ be the space of binary cubic forms
in variables $u$ and $v$,
let $G = \GL_2(\F_q)$, and consider the usual
`twisted action' of $G$ on $V$, given by
\[
(g \circ x)(u, v) = (\det g)^{-1} x((u, v) g).
\]
We write an element of $V$ as
$x=x(u,v)=au^3+bu^2v+cuv^2+dv^3=(a,b,c,d)$.
Let $\disc(x)=b^2c^2+18abcd-4ac^3-4b^3d-27a^2d^2$
be the discriminant of $x$. Then $\disc(gx)=(\det g)^2\disc(x)$.
We say that $x$ is singular if $\disc(x)=0$, or equivalently if
$x$ has a multiple root in $\mathbb P^1$.

The following orbit description is well known.
For a proof, see, e.g., \cite[Section 2]{Wright} or \cite[Section 5.1]{TT_L}.

\begin{proposition}
The action of $G$ on $V$ consists of six orbits,
of which three are singular. We label them 
with the usual symbols
\[
(0), (1^3),(1^21),(111), (21),(3),
\]
where 
here $(0) = \{ 0 \}$, and for the remaining orbits the symbol indicates the degrees and multiplicities of the irreducible
factors of any representative form $x$, thus the first three orbits are singular.
The following table lists a representative,
the number of zeros in ${\mathbb P}^1(\F_q)$
of any element in the orbit,
and the size of the orbit:
% and 
%the size of the stabilizer in $G(\F_q)$ of any point in the orbit.
\[
\begin{array}{ r c c c}
\textnormal{Orbit name}
&\textnormal{Representative}
&\textnormal{Zeros}
&\textnormal{Orbit size}
% & \textnormal{Stabilizer size}
\\ \hline
\calO_{(0)} = \calO_1 & 0 & q+1 & 1 %& (q^2 - 1)(q^2 - q) \\
\\
\calO_{(1^3)} = \calO_2& v^3 &1 & q^2 - 1 %& q^2 - q \\
\\
\calO_{(1^21)} = \calO_3 & uv^2 & 2 & q (q^2 - 1)% & q - 1\\
\\
\calO_{(111)} = \calO_4 & uv(u-v) & 3 & \frac{1}{6} (q^2 - 1)(q^2 - q) %& 6  \\
\\
\calO_{(21)} = \calO_5& v(u^2+a_2uv+b_2v^2) & 1 & \frac{1}{2} (q^2 - 1)(q^2 - q) %& 2 \\
\\
\calO_{(3)} = \calO_6 & u^3+a_3u^2v+b_3uv^2+c_3v^3 & 0 & \frac{1}{3} (q^2 - 1)(q^2 - q) %& 3
\end{array}
\]
Here $u^2+a_2u+b_2$ and  $u^3+a_3u^2+b_3u+c_3\in\F_q[u]$
are respectively any irreducible quadratic and cubic polynomials.
\end{proposition}

We now assume that $\chr(\F_q)\neq3$.
We define a symmetric bilinear form on $V$ by
\begin{equation}\label{eq:sym32-bilinear-form}
[x,x']:=aa'+bb'/3+cc'/3+dd'.
\end{equation}
Then we have $[gx,g^{-T}x']=[x,x']$ and so
$(G,V)$ satisfies Assumption \ref{assmp:dual-identify}.

The following table describes the counts of elements in each orbit
for a variety of subspaces $W_i$.
\[
\small
\begin{array}{r||ccccccc}
\hline
\multicolumn{1}{c||}{\text{Subspace}}
&\co_{(0)}&\co_{(1^3)}&\co_{(1^21)}&\co_{(111)}&\co_{(21)}&\co_{(3)}\\
\hline
W_0=\{(0,0,0,0)\}
&1\\
W_1=\{(0,0,0,*)\}
&1&q-1\\
W_2^\perp=W_2=\{(0,0,*,*)\}
&1&q-1&q(q-1)\\
W_3=\{(0,*,*,0)\}
&1&&2(q-1)&(q-1)^2\\
%W_4=\{(0,0,*,0)\}
%&1&&q-1\\
W_3^\perp=\{(*,0,0,*)\}&-&-&-&-&-&-\\
(q\equiv1\pmod 3)
&1&2(q-1)&&\frac13(q-1)^2&&\frac23(q-1)^2\\
(q\equiv2\pmod 3)
&1&2(q-1)&&&(q-1)^2\\
W_1^\perp=\{(0,*,*,*)\}
&1&|\co_{(1^3)}|/(q+1)&2|\co_{(1^21)}|/(q+1)&3|\co_{(111)}|/(q+1)&|\co_{(21)}|/(q+1)\\
%W_4^\perp=\{(*,0,*,*)\}
%&1&2(q-1)&(2q-1)(q-1)&\frac16(q-1)^2(q-2)&\frac12q(q-1)^2&\frac13(q-1)(q^2-1)\\
V=\{(*,*,*,*)\}
&1&|\co_{(1^2)}|&|\co_{(1^21)}|&|\co_{(111)}|&|\co_{(21)}|&|\co_{(3)}|\\
\hline
\end{array}
\]

Here, the notation $\{(0,0,*,*)\}$ (for example) means the subspace of binary cubic forms
whose first two coefficients are zero and whose latter two are arbitrary.
\begin{remark}
Note that (for example)
we do not literally have $W_2^\perp = W_2$; rather,
these two spaces $W_2=\{(0,0,*,*)\}$ and $W_2^\perp=\{(*,*,0,0)\}$
are
$\GL_2(\F_q)$-equivalent and so the intersections with each $\calO_i$ have the same size.
Similarly the subspace listed for $W_1^\perp$ is in fact a $\GL_2(\F_q)$-transformation of $W_1^\perp=\{(*,*,*,0)\}$, and
in general we will (if the transformations are obvious) apply such transformations in our listings of subspaces
without further comment.
\end{remark}

\begin{remark}
In the literature the alternating form $[x,x']^\sim =da'-cb'/3+bc'/3-ad'$ is sometimes introduced and used instead of our $[x,x']$ in
\eqref{eq:sym32-bilinear-form}; see for example
Shintani \cite{shintani}. Since $[\twtw01{-1}0x,x']=[x,x']^\sim$, these two bilinear forms are essentially the same. Here we choose this $[x,x']$ 
because it is more similar to the forms associated to other representations in later sections.
\end{remark}

The counts above are for the most part trivial to verify,
and so we only describe a few cases explicitly.
The subspace $W_1^{\perp}$ consists of those forms having $[1 : 0]$ as a zero;
since $\GL_2(\F_q)$ acts transitively on $\P^1(\F_q)$,
we have that
$\frac{ |W_1^\perp \cap \calO_i|}{|\calO_i|} \cdot (q + 1)$
is equal to the number of zeros of any $x \in \calO_i$ in $\P^1(\F_q)$.

For $W_3^\perp$, 
let $x=(a,0,0,d)=au^3+dv^3\in W_3^\perp$ with $ad\neq 0$.
Then $x$ is non-singular, with a zero in $\P^1(\F_q)$ if and only if 
$\frac{d}{a} \in( \F_q^\times)^3$.
If $q\equiv -1\mod 3$, then $(\F_q^\times)^3 = \F_q^\times$, so that every
$x$ has a root in $\F_q$. Moreover, the quotient of any two roots of $x$
is a third root of unity, hence not in $\F_q$, so that we have
$x\in\co_{(21)}$ for all $(q - 1)^2$ forms $x$.
If $q\equiv 1\mod 3$,
then $(\F_q^\times)^3$ is the index three subgroup of $\F_q^\times$,
there are $\frac23(q-1)^2$ irreducible $x \in \co_{(3)}$, and the remaining
$\frac13(q - 1)^2$ forms $x$ all factor completely and are in $\co_{(111)}$.

We now use Proposition \ref{prop:A} to determine $M$.
We verify by inspection that the vectors
$\left(|W \cap \calO_i| \right)_i$
for $W \in \{ W_0, W_1, W_2, W_3, W_1^\perp, V \}$ are linearly independent,
and that this set together with
$W_3^\perp$ is closed under taking duals.
The linear algebra is not difficult to carry out by hand,
but we used PARI/GP 
\cite{pari} for this purpose.  
We therefore obtain the matrix $M$,
with a different proof than that previously given by Mori.

\begin{theorem}[Mori \cite{mori}]\label{thm:sym32-A}
Suppose $\chr(\F_q)\neq3$.
We have
\[
M=\frac{1}{q^4}
\begin{bmatrix}
1&q^2-1&q^3-q
&(q^2-1)(q^2-q)/6&(q^2-1)(q^2-q)/2&(q^2-1)(q^2-q)/3\\
1&-1&q^2-q
&q(q-1)(2q-1)/6&-q(q-1)/2&-q(q^2-1)/3\\
1&q-1&q^2-2q&-q(q-1)/2&-q(q-1)/2&0\\
1&2q-1&-3q&q(5\pm q)/6&-q(-1\pm q)/2&q(-1\pm q)/3\\
1&-1&-q&-q(-1\pm q)/6&q(1\pm q)/2&-q(-1\pm q)/3\\
1&-q-1&0&q(-1\pm q)/6&-q(-1\pm q)/2&q(2\pm q)/3\\
\end{bmatrix},
\]
where the signs $\pm$ appearing in right-lower $3$-by-$3$ entries
are according as $q\equiv\pm1\pmod 3$.
\end{theorem}

We derive two formulas as corollaries to Theorem \ref{thm:sym32-A}.
The first one below was previously given and used
in \cite{TT_L} to establish
an analogue of the Ohno-Nakagawa formula for the
`divisible zeta function'. In a companion paper \cite{TT_leveldist},
we use this to study almost-prime
cubic field discriminants.

\begin{corollary}\label{cor:sym32-Psi}
For a finite field $\F_q$ of characteristic not equal to $3$,
write $\Psi_q(x)$ for the characteristic function of
singular binary cubic forms over $\F_q$.
We have
\begin{equation}\label{eq:Psi_fourier_binary}
\widehat{\Psi_q}(x)=
\begin{cases}
q^{-1}+q^{-2}-q^{-3} & x=0,\\
q^{-2}-q^{-3} & x\neq0, \Disc(x)=0\\
-q^{-3} & \Disc(x)\neq0,\\
\end{cases}
\end{equation}
\end{corollary}
This is immediate from Theorem \ref{thm:sym32-A},
because $\Psi_q=e_1+e_2+e_3$
and thus $\widehat{\Psi_q}=\widehat{e_1}+\widehat{e_2}+\widehat{e_3}$.
Another consequence of Theorem \ref{thm:sym32-A}
is an explicit formula of the Fourier transform of $\chi(\Disc(x))$,
where $\chi$ is the quadratic character on $\F_q$. This result is contained within Denef and Gyoja's
main theorem \cite{DG} and in this case we obtain a simpler proof.

\begin{corollary}\label{cor:sym32-quadratic-twist}
Assume $p=\chr(\F_q)\neq2,3$, and let
$\chi\colon\F_q^\times\rightarrow\{\pm1\}$
be the unique quadratic character.
We use the usual convention
$\chi(0)=0$. We have
\begin{equation}\label{eq:sym32-fourier-quadratic}
\frac{1}{q^4}\sum_{x\in V}\chi\left(\Disc(x)\right)
\exp\left(2\pi\sqrt{-1}\cdot \frac{{\rm Tr}_{\F_q/\F_p}([x,y])}{p}\right)
=\frac{1}{q^2}\cdot\chi\left(\Disc^\ast(y)\right),
\end{equation}
where $\Disc^\ast(y)=-\Disc(y)/27$
is the normalized invariant for the dual space $V^\ast$.
\end{corollary}
\begin{proof}
If $x\in V$ is non-singular,
then $\chi\left(\Disc(x)\right)=1$ or $-1$
according as $x\in\co_{(111)}\cup\co_{(3)}$ or $x\in\co_{(2)}$.
Hence $\chi\left(\Disc(\cdot)\right)=e_4-e_5+e_6$ and
the left hand side of \eqref{eq:sym32-fourier-quadratic}
is $(\widehat{e_4}-\widehat{e_5}+\widehat{e_6})(y)$.
By Theorem \ref{thm:sym32-A}, we have
\[
\widehat{e_4}-\widehat{e_5}+\widehat{e_6}
=\pm q^{-2}\left(e_4-e_5+e_6\right)
=\pm q^{-2}\cdot\chi\left(\Disc(\cdot)\right)
\]
where the sign is according as $q\equiv\pm1\pmod 3$.
Since $\chi(-27)=\chi(-3)=\pm1$,
where again the sign is according as $q\equiv\pm1\pmod 3$,
we have the result.
\end{proof}

We illustrate an application of this formula to the
functional equation of the Shintani zeta function.
Let $n\neq 1$ be a square-free integer coprime to $6$,
and $\chi$ be the unique
primitive quadratic Dirichlet character modulo $n$.
By the Chinese remainder theorem, the similar formula
for \eqref{eq:sym32-fourier-quadratic} is true for $\chi$.
For each sign, we define
\[
\xi_\pm(s,\chi)
:=\sum_{\substack{x\in \gl_2(\Z)\backslash \sym^3(\Z^2)\\\pm\Disc(x)>0}}
\frac{\chi(\Disc(x))}{|\Stab(x)|}|\Disc(x)|^{-s}.
\]
This is a quadratic twist of the zeta function
introduced and studied by Shintani \cite{shintani}.
Note that we put $\chi(\Disc(x))=0$
if $\Disc(x)$ is not coprime to $n$.
It is shown in \cite{TT_L} that
these Dirichlet series $\xi_\pm(s,\chi)$
enjoy analytic continuation as entire functions to
the whole complex plane
and satisfy functional equations.
Corollary \ref{cor:sym32-quadratic-twist} may then be
applied to describe this functional equation explicitly.
Write
\begin{align*}
\xi_{\rm add}(s,\chi)&:=3^{1/2}\xi_+(s,\chi)+\xi_-(s,\chi),\\
\xi_{\rm sub}(s,\chi)&:=3^{1/2}\xi_+(s,\chi)-\xi_-(s,\chi),
\end{align*}
and
\begin{align*}
\Lambda_{\rm add}(s,\chi)&:=
\left(\frac{432n^4}{\pi^4}\right)^{s/2}
		\Gamma\left(\frac s2\right)\Gamma\left(\frac s2+\frac12\right)
		\Gamma\left(\frac{s}{2}-\frac1{12}\right)
		\Gamma\left(\frac{s}{2}+\frac1{12}\right)
	\xi_{\rm add}(s,\chi),\\
\Lambda_{\rm sub}(s,\chi)&:=
\left(\frac{432n^4}{\pi^4}\right)^{s/2}
		\Gamma\left(\frac s2\right)\Gamma\left(\frac s2+\frac12\right)
		\Gamma\left(\frac{s}{2}+\frac5{12}\right)
		\Gamma\left(\frac{s}{2}+\frac7{12}\right)
		\xi_{\rm sub}(s,\chi).
\end{align*}
Then similarly to \cite{O, TT_L},
Corollary \ref{cor:sym32-quadratic-twist}
combined with Datskovsky-Wright's diagonalization \cite{DW2}
and Nakagawa's dual identity \cite{N}
enables us to write the functional equation
of $\xi_\pm(s,\chi)$ in a self dual form:
\begin{corollary}\label{cor:ohno-nakagawa}
We have
\[
\Lambda_{\rm add}(s,\chi)=\Lambda_{\rm add}(1-s,\chi),
\qquad
\Lambda_{\rm sub}(s,\chi)=\Lambda_{\rm sub}(1-s,\chi).
\]
\end{corollary}

\section{$\Sym^2(2)$}\label{sec:sym22}
We now turn our attention to the easier case of binary quadratic forms.
Let $V = \Sym^2(\F_q^2)$ be the space of binary quadratic forms
in variables $u$ and $v$,
together with the action of $G=\GL_1(\F_q)\times\GL_2(\F_q)$ given by
\begin{equation}\label{eqn:act_bqf}
(g_1, g_2) \cdot x(u, v) = g_1 x((u, v) g_2).
\end{equation}
We write an element of $V$ as $x=x(u,v)=au^2+buv+cv^2=(a,b,c)$.
We let $\disc(x)=b^2-4ac$, and say $x$ is singular if $\disc(x)=0$.

Let $\chr(\F_q)\neq2$.
As usual, we identify $x$ with the two-by-two symmetric matrix
$A=\begin{bmatrix}a&{b/2}\\{b/2}&c\end{bmatrix}$, with $\disc(x)=-4\det A$.
Then the action of the $\gl_2(\F_q)$ part is given by
$(g_2,A)\mapsto g_2Ag_2^T$, while the $\gl_1(\F_q)$ part
acts by scalar mulitplication.

$V$ consists of four $G$ orbits, which we enumerate as follows.
(Rank is the rank as a symmetric matrix of any element in the orbit.)
$\co_1$ and $\co_2$ are the singular orbits.
\[
\begin{array}{crccc}
\text{Symbol}	&\text{Orbit name}	&\text{Representative}
	&\text{Rank}
	&\text{Orbit size}
%&	\text{Stabilizer size}
\\
\hline
(0)	&\co_{(0)}=\co_1	&0	&0	&1
%&(q-1)(q^2-1)(q^2-q)\\
\\
(1^2)	& \co_{(1^2)}=\co_2	&v^2	&1	&q^2-1
%&(q-1)(q^2-q)\\
\\
(11)	& \co_{(11)}=\co_3	&uv	&2	&\frac12q(q^2-1)
%&2(q-1)^2\\
\\
(2)	& \co_{(2)}=\co_4	&u^2-lv^2&2	&\frac12q(q-1)^2
%&2(q^2-1)\\
\end{array}
\]
Here
$l\in\F_q^\times$ denotes an arbitrary non-square element.
The proof is elementary and
we omit the details.
Note that the factor of $\GL_1$ is included to ensure that the
$G$-orbits of $v^2$ and $l v^2$ coincide.

We define a symmetric bilinear form on $V$ by
\begin{equation}\label{eq:sym22-bilinear-form}
[x,x']:=aa'+bb'/2+cc'.
\end{equation}
Then we have $[gx,g^{-T}x']=[x,x']$ and so
$(G,V)$ satisfies Assumption \ref{assmp:dual-identify}.

The counts of the $W \cap \calO_i$
for the following subspaces $W$ are immediately verified:
\[
\begin{array}{r||ccccc}
\hline
\text{Subspace}
&\co_{(0)}&\co_{(1^2)}&\co_{(11)}&\co_{(2)}\\
\hline
\{(0,0,0)\}
&1\\
\{(0,0,*)\}&1&q-1\\
\{(0,*,0)\}&1&&q-1\\
\{(0,*,*)\}&1&q-1&q^2 - q\\
\{(*,0,*)\}&1&2q-2&\frac12(q - 1)^2&\frac12(q - 1)^2\\
\{(*,*,*)\}
&1&|\co_{(1^2)}|&|\co_{(11)}|&|\co_{(2)}|\\
\hline
\end{array}
\]
Therefore by Proposition \ref{prop:A},
we immediately obtain the following.
(We in fact do not require the above counts for
$\{(0,*,0)\}$ and $\{(*,0,*)\}$.)
\begin{theorem}\label{thm:sym22-A}
Suppose $\chr(\F_q)\neq2$.
We have
\[
M=
\frac{1}{q^3}\cdot
\begin{bmatrix}
1&q^2-1&\frac12 q(q^2-1) &\frac12q(q-1)^2\\
1&-1&\frac12(q^2-q) &-\frac12(q^2-q)\\
1&q-1&-q&0\\
1&-(q+1)&0&q\\
\end{bmatrix}.
\]
\end{theorem}

We can now provide an example where a complete analogue of Corollary \ref{cor:sym32-quadratic-twist}
is not guaranteed by Denef-Gyoja, and indeed is not true:
\begin{remark}\label{rem:sym22}
Let $\chi$ be the quadratic character on $\F_q^\times$.
Then since $\chi(\Disc(\cdot))=e_3-e_4$, its Fourier transform is
\[
\widehat{e_3}-\widehat{e_4}=(q^{-1}-q^{-2})(e_1+e_2)-q^{-2}(e_3+e_4).
\]
This function does not vanish on the singular set.
%In particular, a formula like \eqref{eq:sym32-fourier-quadratic}
%does {\em not} hold for $V=\sym^2(2)$.
\end{remark}

\section{$\Sym^2(3)$}\label{sec:sym23}
Now, let $V = \Sym^2(\F_q^3)$ be the space of ternary quadratic forms in
variables $u$, $v$ and $w$,
together with the action of $G=\GL_1(\F_q)\times\GL_3(\F_q)$ given by
\[
(g_1, g_3) \cdot x(u, v, w) = g_1 x((u, v, w) g_3).
\]
We write an element of $V$ as
\[
x=x(u,v,w)=au^2+buv+cuw+dv^2+evw+fw^2=(a,b,c,d,e,f).
\]
We again assume $\chr(\F_q)\neq2$.
As in the previous section,
we identify $V$ as the space of symmetric matrices of size three.
Hence $x$ is identified with
$A=\begin{bmatrix}a&b/2&c/2\\b/2&d&e/2\\c/2&e/2&f\end{bmatrix}$.
We define $\disc(x)=-4\det(A)$, and say $x$ is singular if $\disc(x)=0$.

\begin{proposition}\label{prop:orbit_ternary}
Assume $\chr(\F_q)\neq 2$, and let $l\in\F_q^\times$
denote an arbitrary non-square element.
The action of $G$ on $V$ has five orbits,
of which four are singular. For each orbit, 
the table below lists orbital representatives,
%the size of the stabilizer in $G$ of any point in the orbit,
the number of zeros in $\P^2(\F_q)$ and the rank as a symmetric matrix
of any element in the orbit,
and the size of the orbit.
\[
\begin{array}{ c r c c c c}
\textnormal{Symbol}
&\textnormal{Orbit name}
&\textnormal{Representative}
%& \textnormal{Stabilizer size}
& \textnormal{Zeros}
& \textnormal{Rank}
& \textnormal{Orbit size}
\\ \hline
(0)
& \co_{(0)}=\co_1
& 0
%& (q - 1)(q^3 - 1)(q^3 - q)(q^3 - q^2)
& q^2 + q + 1
&0
& 1
\\
(1^2)
& \co_{(1^2)}=\co_2
& u^2
%& (q - 1)(q^3 - q)(q^3 - q^2)
& q + 1
&1
& q^3 - 1
\\
(11)
& \co_{(11)}=\co_3
& uv
%& 2(q - 1)^2(q^3 - q^2)
& 2q + 1
&2
&  q(q^3 - 1)(q + 1)/2
\\
(2)
& \co_{(2)}=\co_4
& u^2 - l v^2
%& 2(q^2 - 1)(q^3 - q^2)
& 1
&2
& q(q^3 - 1)(q - 1)/2
\\
\textnormal{ns}
& \co_{\rm ns}=\co_5
& u^2 - vw
%& (q - 1)(q^3 - q)
& q + 1
&3
& q^2(q^3 - 1) (q - 1)
\\
\end{array}
\]
Here $l\in\F_q^\times$ is a non-square element.
\end{proposition}

\begin{proof}
This is well known (see e.g. \cite{elkies}),
and for the sake of completeness we include a proof.

To prove that there are five orbits, start with an arbitrary $x \in V$ and complete the square to
get rid of any off-diagonal terms. If $x$ is not $G$-equivalent to one of the first four representatives it must be of the form
$\lambda_1 u^2 - \lambda_2 v^2 - \lambda_3 w^2$
with $\lambda_1 \lambda_2 \lambda_3 \neq 0$, and indeed with
$\lambda_1 = 1$ after 
multiplying by $\lambda_1^{-1} \in \GL_1(\F_q)$.
Now, if $\lambda_2$ or $\lambda_3$ is a square element,
then this form is visibly equivalent
to $u^2 - vw$. Otherwise, by a suitable $\GL_3(\F_q)$ translation we may assume that $\lambda_2 = \lambda_3$. Moreover, for each 
$\alpha \in \F_q^{\times}$ we have that $y^2 + z^2$ is equivalent to 
$(v + \alpha w)^2 + (w - \alpha v)^2 = (1 + \alpha^2) (v^2 + w^2)$; as $1 + \alpha^2$ cannot be a square element for every $\alpha$,
we see that $\lambda_2(v^2 + w^2)$ is equivalent to $v^2 + w^2$.

Since the number of $\F_q$-rational zeros and the rank are invariants of the orbits, the orbits listed above are seen to all be distinct.
 
The first two orbit sizes are very easy to compute; 
the next two are most easily computed by observing that the stabilizer size
is $q^3 - q^2$ times the analogous stabilizer size in $\Sym^2(2)$, as any $g_3 \in \GL_3$ in the stabilizer may send
$w$ to any $a_1 u + a_2 v + a_3 w$ with $a_3 \neq 0$.
The final orbit size is most easily computed by subtracting the
first four orbit sizes from $q^6$.
\end{proof}

We define a symmetric bilinear form on $V$ by
\begin{equation}\label{eq:sym23-bilinear-form}
[x,x']:=aa'+bb'/2+cc'/2+dd'+ee'/2+ff'.
\end{equation}
Then we have $[gx,g^{-T}x']=[x,x']$ and so
$(G,V)$ satisfies Assumption \ref{assmp:dual-identify}.

We come now to the computations of $|W \cap \calO_i|$ for suitable $W$.
In the table below, we write
%$(a, b, c, d, e, f)$ for the quadratic form $au^2 + buv + cuw + dv^2 + evw + fw^2$, 
$\times \alpha$ as a shorthand for $\alpha |\calO_i|$.
\[
\begin{array}{r||cccccc}
\hline
\multicolumn{1}{c||}{\text{Subspace}}
&\co_{(0)}&\co_{(1^2)}&\co_{(11)}&\co_{(2)}&\co_{\ns}\\
\hline
W_0=\{(0,0,0,0,0,0)\}
&1\\
W_1=\{(0,0,0,0,0,*)\}
&1&q-1\\
W_2=\{(0,0,0,*,0,*)\}
&1&2q-2&\frac{(q - 1)^2}{2}&\frac{(q - 1)^2}{2}\\
W_3=\{(0,0,*,*,*,*)\}
&1&q^2-1&\frac12q(3q^2-2q-1)&\frac12q(q-1)^2& q^2(q-1)^2\\
W_1^\perp=\{(0,*,*,*,*,*)\} &
1&
\times \frac{q + 1}{q^2 + q + 1}&
\times \frac{2q + 1}{q^2 + q + 1}&
\times \frac{1}{q^2 + q + 1}&
\times \frac{q + 1}{q^2 + q + 1}
\\
W_2^\perp=\{(0,*,*,0,*,*)\} &
1&
\times \frac{q + 1}{q^2 + q + 1} \cdot \frac{q}{q^2 + q}&
\times \frac{2q + 1}{q^2 + q + 1} \cdot \frac{2q}{q^2 + q}&
0&
\times \frac{q + 1}{q^2 + q + 1} \cdot \frac{q}{q^2 + q}
\\
V=\{(*,*,*,*,*,*)\}
&1&\times 1&\times 1&\times 1&\times 1\\
\hline
\end{array}
\]
The map
\[
du^2 + euv + fv^2 \longmapsto 0u^2 + 0uv + 0uw + dv^2 + evw + fw^2
\]
is an embedding $\Sym^2(2) \rightarrow \Sym^2(3)$,
and as $a = b = c = 0$ for $W_0$, $W_1$,
and $W_2$ the counts coincide with those previously
given for $\Sym^2(2)$. 
For $W_3$, let $x=cuw+dv^2+evw+fw^2\in W_3$.
This is non-singular if and only if $cd\neq0$.
Hence there are $q^2(q-1)^2$ elements of $\co_\ns$.
The count for $c=0$ follows from the whole of $\Sym^2(2)$,
and for $d=0$ follows immediately.
For $W_1^\perp$ and $W_2^\perp$,
we are counting the number of elements of each $\calO_i$ with,
respectively, having $[1:0:0]$ as a zero,
and having $[1:0:0]$ and $[0:1:0]$ as zeros.
As $\GL_3(\F_q)$ acts $4$-transitively on $\P^2(\F_q)$,
the proportion of elements of each $\calO_i$ in
$W_1^\perp$ and $W_2^\perp$ is respectively
$\frac{ \# Z_x(\F_q) }{ \# \P^2(\F_q) }$ and
$\frac{ \# Z_x(\F_q)  (\# Z_x(\F_q) - 1) }
{ \# \P^2(\F_q)  (\# \P^2(\F_q) - 1) }$
for any $x\in\co_i$
where $Z_x\subset\P^2$ is the conic defined by $x$,
%conic section\footnote{describe more precisely} $x$ in the $G$-orbit,
and these quantities $\# Z_x(\F_q)$ were enumerated above.

The above vectors span $\R^5$,
and as before by Proposition \ref{prop:A} we conclude:
\begin{theorem}\label{thm:sym23-A}
Suppose $\chr(\F_q)\neq2$.
We have
\[
M=\frac{1}{q^6}
\begin{bmatrix}
1&
q^3-1&
q(q+1)(q^3-1)/2&
q(q-1)(q^3-1)/2&
q^2(q-1)(q^3-1)\\
1&
-1&
q(q+1)(q^2-1)/2&
-q(q-1)(q^2+1)/2&
-q^2(q-1)\\
1&
q^2-1&
q(q^2-2q-1)/2&
q(q-1)^2/2&
-q^2(q-1)\\
1&
-q^2-1&
q(q^2-1)/2&
q(q^2+1)/2&
-q^2(q-1)\\
1&
-1&
-q(q+1)/2&
-q(q-1)/2&
q^2
\end{bmatrix}.
\]
\end{theorem}

\section{$2\otimes\Sym^2(2)$}\label{sec:2sym22}

Next we investigate the space $V$ of pairs of binary quadratic forms.
As before, we assume $\chr(\F_q)\neq2$.
Let $V:=\F_q^2\otimes \sym^2(\F_q^2)$ be the space
of pairs of binary quadratic forms.
We write an element of $V$ as follows:
\begin{equation}\label{eqn:def_2sym2}
x=(A,B)=(au^2+buv+cv^2,du^2+euv+fv^2)=
\begin{bmatrix}a&b&c\\ d&e&f\end{bmatrix}
\end{equation}
Here, $A=\twtw{a}{b/2}{b/2}{c}$ and 
$B=\twtw{d}{e/2}{e/2}{f}$
are
the symmetric matrices representing the respective binary quadratic forms.
Let $G_1=G_2=\gl_2(\F_q)$ and $G:=G_1\times G_2$.
The group action of $G$ is defined by
\[
\left(\twtw \alpha\beta\gamma\delta, g_2\right) \circ (A, B)
=(\alpha g_2 A g_2^T + \beta g_2 B g_2^T,
\gamma g_2 A g_2^T + \delta g_2 B g_2^T).
\]
To investigate this representation,
it is useful to associate to each $x\in V$ its
{\em quadratic resolvent} $r_x\in\sym^2(\F_q^2)$,
defined by
\[
r_x(u,v):=-4\det(Au+Bv)=(bu+ev)^2-4(au+dv)(cu+fv).
\]
Then we can see easily from the definition that
$r_{(g_1,g_2)\circ x}=(\det g_2)^2(g_1\circ r_x)$,
where the action of $g_1 \in \GL_2(\F_q)$ on $\sym^2(\F_q)$
is as in (the $\GL_2(\F_q)$ component of)
\eqref{eqn:act_bqf}.
We define $\Disc(x):=\Disc(r_x)$,
and say $x$ is singular if $\Disc(x)=0$.

We first study its orbit decomposition.
\begin{proposition}\label{prop:orbit-2sym22}
Assume $\chr(\F_q)\neq2$. There are seven orbits over $\F_q$, of which the first five are singular.
For each orbit, the following table lists orbital representatives,
the rank, the label of the quadratic resolvent of any element in the orbit,
and the size of the orbit.
Here the {\em rank} of $x=(A,B)\in V$ 
is the rank of the matrix in \eqref{eqn:def_2sym2}.
\[
\begin{array}{rcccc}
\text{\rm Orbit name}
&\text{\rm Representative}
&\text{\rm Rank}
&\text{\rm Resolvent}
&\text{\rm Orbit size}
\\
\hline
\co_{(0)}=\co_1
&%\pb000000=
(0,0)
&0
&(0)
&1
\\
\odr=\co_2
&%\pb000001=
(0,v^2)
&1
&(0)
&(q - 1)(q+1)^2
\\
\ods=\co_3
&%\pb000010=
(0,uv)
&1
&(1^2)
&q(q-1)(q+1)^2/2
\\
\odi=\co_4
&%\pb00010{-l}=
(0,u^2-lv^2)
&1
&(1^2)
&q(q-1)^2(q+1)/2
\\
\ocs=\co_5
&%\pb001010=
(v^2,uv)
&2
&(1^2)
&q(q^2 - 1)^2
\\
\ots=\co_6
&%\pb010101\slash \pb001100%=
(v^2,u^2)
&2
&(11)
&(q^3 - q)^2/2
\\
\oti=\co_7
&%\pb01010l%=
(uv,u^2+lv^2)
&2
&(2)
&(q^2 - q)^2 (q^2 - 1)/2
\\
\end{array}
\]
Here $l\in\F_q$ is a non-square element.
\end{proposition}
\begin{remark}
The subscripts $D$ and $C$ indicate respectively that
$x$ is doubled (i.e., rank $1$) or has a common component.
The subscript $B$ (binary) is chosen to be consistent with
the quartic case $V=2\otimes\sym^2(3)$.
\end{remark}
\begin{proof}
We first show that any $x=(A,B)=\pb abcdef\in V$ is $G$-equivalent to
one of the seven elements above.
We write $x\sim x'$ if $x,x'\in V$ are in the same $G$-orbit.
If $A$ and $B$ are linearly dependent,
then by the action of $G_1$,
we may let $A=0$. By our results for $\sym^2(2)$,
$x$ is equivalent to one of the first four elements.
So suppose $A$ and $B$ are linearly independent.
If the left side 2-by-2 matrix $\twtw abde$
of $x$ is invertible (as a matrix),
then by a $G_1$ translation,
we can move
$x$ to $\pb 01{c'}1{e'}{f'}$.
By a $G_2$ translation,
we move this to $\pb 0101{e''}{f''}$ and
then again by $G_1$
we further erase $e''$,
and thus assume $x=(uv,u^2+fv^2)$.
Observe that $(uv,u^2+fv^2)\sim (tuv,u^2+ft^2v^2)\sim (uv,u^2+ft^2v^2)$
for $t\in\F_q^\times$, $(uv,u^2+fv^2)\sim(uv,u^2+lv^2)$ if $f$ is non-square.
If $f$ is square, $(uv,u^2+fv^2)\sim (uv,u^2+v^2)$
and is further $\sim (u^2+v^2+2uv,u^2+v^2-2uv)\sim (v^2,u^2)$.
Suppose $\twtw abde$ is not invertible (it cannot be zero).
By $G_1$,
we may let $x$ be either $\pb 00c01f$ or $\pb 00c1ef$.
In the latter case we use $\twtw 10{-e/2}1\in G_2$ to move it
to $\pb 00c10{f'}$.
Since $A$ and $B$ are independent, $c\neq0$ and so
$x$ is equivalent to $\pb001010=(v^2,uv)$ or $\pb001100=(v^2,u^2)$,
respectively.

We next prove that the seven elements are in different orbits.
We compare their ranks and the quadratic resolvent:
if $x,y\in V$ are in the same orbit, then
${\rm rank}(x)={\rm rank}(y)$ and also
$r_x,r_y\in\sym^2(\F_q^2)$ are of the same type.
To conclude it is enough to check that
$(0,v^2)$ and $(0,u^2-lv^2)$ are not in the same orbit.
Our assertion follows from the results of $\sym^2(2)$,
because $x=(0,B)$ and $x'=(0,B)$ with $B\neq0$, $B'\neq0$ are
in the same $G$-orbit iff $B,B'\in\sym^2(\F_q^2)$ lie in the same
$\gl_1(\F_q)\times\gl_2(\F_q)$-orbit.

Thus we have shown that there are exactly seven orbits, and 
the orbit counts will be proved later in this section.
\end{proof}

Writing the symmetric bilinear form \eqref{eq:sym22-bilinear-form}
on $\sym^2(\F_q^2)$ by $[A,A']'$, we define
a symmetric bilinear form on $V$ by
\begin{equation}\label{eq:2sym22-bilinear-form}
[(A, B), (A', B')] = [A, A']' + [B, B']'.
\end{equation}
Then we have $[gx,g^{-T}x']=[x,x']$
(recall that $(g_1, g_2)^T := (g_1^T, g_2^T)$)
and so
$(G,V)$ satisfies Assumption \ref{assmp:dual-identify}.

We come now to our orbit counts.
Anticipating the more difficult computations in the quartic case,
we introduce a new method for the computations
which enables us to work inductively (and rather systematically).
By subtracting the results of previously handled computations,
we may assume that certain
of the coordinates are nonzero, and then apply
$G$-transformations to obtain a map to a set $Y$ for which the 
$|Y \cap \co_i|$ are more easily counted.
For $X\subset V$, we find it convenient to use the same symbol
$X$ to denote the vector $(X\cap\co_i)_i\in\R^7$.
For example,
\[
(X\cup Y)=X+Y-(X\cap Y)
\]
is an identity in the vector space $\R^7$,
representing the inclusion-exclusion principle.

For $i,j\in\{0,1,2,3\}$, let $\sbs ij$
be the subspace consisting of pairs of forms $(A, B)$ such that
the last $i$ entries of $A$ and the last $j$ entries of $B$ are arbitrary,
and other entries are $0$.
For example,
\[
\sbs00=\{0\},
\quad
\sbs13=\left\{\begin{bmatrix}0&0&*\\ *&*&*\end{bmatrix}\right\},
\quad
\sbs22=\left\{\begin{bmatrix}0&*&*\\ 0&*&*\end{bmatrix}\right\},
%\sbs22=\pb 0**0**,
\quad
\sbs33=\left\{\begin{bmatrix}*&*&*\\ *&*&*\end{bmatrix}\right\}
=V.
\]
Of course some coordinate subspaces of $V$,
such as $\left\{\pb 00\ast\ast\ast0\right\}$, are not of this form,
but here it suffices to consider only such subspaces.
We count $\sbs ij\cap \co_k$ for all
$0\leq i\leq j\leq 3$
(there are ten of them),
because we need this result when studying $2\otimes\sym^2(3)$
in the next section.
(However, as we observe when proving Theorem \ref{thm:2sym22-A},
counts for a certain seven of $\sbs ij$ are enough
to determine $A$.)
Note that $\eta\cdot \sbs ij^\perp=\sbs{3-j}{3-i}$
where $\eta=\left(\twtw 0110,\twtw0110\right)\in G$,
so that $\sbs ij^\perp =\sbs{3-j}{3-i}$
(not literally as subspaces of $V$, but rather
as in the sense described previously).

If $1\leq i<j$, then let $\sbs ij^\times$
be the elements in $\sbs ij$
whose first $*$'s are non-zero in each row.
For example,
\[
\sbs12^\times=
\left\{\begin{bmatrix}0&0&c\\ 0&e&*\end{bmatrix}\ \vrule\ c,e\neq0\right\},
\qquad
\sbs13^\times=
\left\{\begin{bmatrix}0&0&c\\ d&*&*\end{bmatrix}\ \vrule\ c,d\neq0\right\}.
%\sbs13^\times=\left\{\pb 00cdef\mid c,d\neq0\right\}
\]
Then by inclusion-exclusion,
\begin{equation}\label{eq:wij}
\sbs ij
=\sbs ij^\times +\sbs {i-1}j+\sbs i{j-1} -\sbs{i-1}{j-1}.
\end{equation}
If $2\leq i=j$, then let $\sbs ii^\times$
be the elements $x$ in $\sbs ii$
whose leftmost 2-by-2 matrix $\twtw \ast\ast\ast\ast$ 
of $x$ is invertible.
For example,
\[
\sbs22^\times=
\left\{\begin{bmatrix}0&b&c\\ 0&e&f\end{bmatrix}\ \vrule\ bf-ce\neq0\right\},
\\
%sbs22^\times=\left\{\pb 0bc0ef\mid bf-ce\neq0\right\},
\qquad
\sbs33^\times=
\left\{\begin{bmatrix}a&b&*\\ d&e&*\end{bmatrix}\ \vrule\ ae-bd\neq0\right\},
%\sbs33^\times=\left\{\pb abcdef\mid ae-bd\neq0\right\}.
\]
Then in this case we have
\begin{equation}\label{eq:wii}
\sbs ii
=\sbs ii^\times +(q+1)\cdot \sbs {i-2}i-q\cdot \sbs{i-2}{i-2}.
\end{equation}
To verify this, let $W'_{[i, i]}$ denote the subset of $W_{[i, i]}$
whose leftmost two-by-two matrix $M$ is of rank $1$, and $W'_{[i-2,i]}=W'_{[i, i]}\cap W_{[i-2, i]}$;
it suffices to describe an orbit preserving map 
$W'_{[i, i]} \rightarrow W'_{[i - 2, i]}$
which is precisely $(q + 1)$-to-one.
This is most easily described as a bijection
$\P^1(\F_q) \times W'_{[i - 2, i]} \rightarrow W'_{[i, i]}$:
we map $([\mu : 1], x)$ to $(\twtw 1 \beta 0 1, \twtw 1001) \cdot x$ and
$([1 : 0], x)$ to $(\twtw 0 1 1 0, \twtw 1001) \cdot x$.

Hence provided that we have counted for smaller $i$ and $j$,
our count of each $\sbs ij$ is reduced to that of $\sbs ij^\times$,
by \eqref{eq:wij} or \eqref{eq:wii}.
Our induction process may be illustrated in the following diagram:
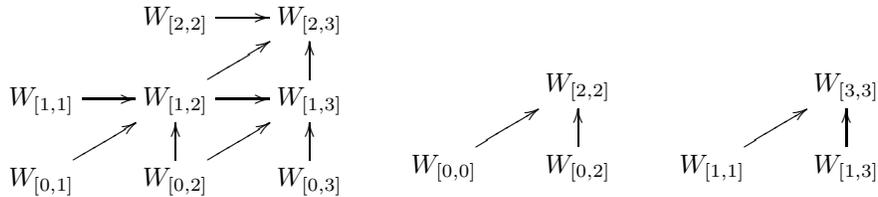
\[
\xymatrix@R=5mm@C=7mm{
&\sbs 22\ar[r]
&\sbs 23
\\
\sbs 11\ar[r]
&\sbs 12\ar[ur]\ar[r]
&\sbs 13\ar[u]
\\
\sbs 01\ar[ur]
&\sbs 02\ar[u]\ar[ur]
&\sbs 03\ar[u]
}
\qquad
\xymatrix@R=5mm@C=7mm{
\\
&\sbs 22
\\
\sbs 00\ar[ur]
&\sbs 02\ar[u]
}
\qquad
\xymatrix@R=5mm@C=7mm{
\\
&\sbs 33
\\
\sbs 11\ar[ur]
&\sbs 13\ar[u]
}
\]
At the beginning we need to know the counts
for $\sbs 0j$ for all $j$ and $\sbs11$.
The former is contained in our results for $\sym^2(2)$,
while the latter is trivial. We have:
\[
\begin{array}{r||ccccccc}
\hline
\text{Subspace}
&\co_0&\odr&\ods&\odi&\ocs&\ots&\oti\\
\hline
\sbs00
&1\\
\sbs01
&1&q-1\\
\sbs02
&1&q-1&q(q-1)\\
\sbs03
&1&q^2-1&\frac12q(q^2-1)&\frac12q(q-1)^2\\
\sbs11
&1&q^2-1\\
\hline
\end{array}
\]
We now examine 
$\sbs12^\times,
\sbs13^\times,
\sbs22^\times,
\sbs23^\times
$
and
$\sbs33^\times$.
For $\sbs 13^\times$, we have
\begin{equation}\label{eqn:w13}
\sbs 13^\times
=(q-1)^2\cdot\pb 0011**
=q^2(q-1)^2\cdot\pb 001100
\end{equation}
so that all of these elements are in $\ots$.
The reduction \eqref{eqn:w13} is proved as follows.
There is a bijection $\F_q^{\times} \times \F_q^{\times} \times \{ \pb 0011** \} \rightarrow W_{[1, 3]}^{\times}$, 
given by $(\lambda, \mu, \pb 0011ab) \rightarrow
({\twtw \lambda 0 0 \mu}, \twtw 1001) \cdot
 \pb 0011ab = 
\pb 00\lambda\mu{\mu a}{\mu b}$, and this yields the first equality. Then, there is a bijection
$\F_q \times \{ \pb 00110* \}  \rightarrow \{ \pb 0011** \}$, given by
$(\lambda, \pb 00110a) \rightarrow
(\twtw 1001, {\twtw 1 0 \lambda 1} ) \cdot \pb 00110a = \pb 0011{2\lambda}{\lambda^2 + a}$.
Finally, there is a bijection
$\F_q \times \{ \pb 001100 \}  \rightarrow \{ \pb 00110* \}$, given by
$(\lambda, \pb 001100) \rightarrow
(\twtw 10\lambda1, {\twtw 1 0 0 1} ) \cdot \pb 001100 = \pb 00110\lambda$.

Put together these three bijections prove \eqref{eqn:w13}; note that each factor of $\F_q$ or $\F_q^{\times}$
determined an element of $G(\F_q)$. Each of these three
steps illustrates a reduction which we will use quite frequently in our analysis, and we will
usually leave similar such verifications to the reader.

For $\sbs 33^\times$, we can similarly prove that
\[
\sbs 33^\times
=(q^2-1)(q^2-q)\cdot\pb 01*10*
=(q^2-1)(q^2-q)\cdot\pb 0101**
=q(q^2-1)(q^2-q)\cdot\pb 01010*;
\]
the
second equality may be verified by observing that for each $\lambda\in\F_q$
$(\twtw 1001, \twtw10\lambda1)$ gives a bijective map from
$\pb01{-\lambda}10*$ to $\pb0101{2\lambda}*$.

The element $\pb 01010f$ is, according as
$f$ is zero, a quadratic residue or a quadratic non-residue,
in $\ocs$, $\ots$ or $\oti$, respectively.
The remaining cases are treated similarly and easily.
We summarize our argument in the following table.
In each row the orbit counts for $W^{\times}$ are equal to the
multiplier times those for $Y$, and those for $Y$ are listed in the middle
columns.
\[
\begin{array}{r||c|ccc|l}
\hline
\text{Subset}&
Y &
\ocs&\ots&\oti
&\multicolumn{1}{|c}{\text{multiplier}}
\\
\hline
\sbs12^\times
&\pb001010
&1&&&\times q(q-1)^2
\\
\sbs13^\times
&\pb001100
&&1&&\times q^2(q-1)^2
\\
\sbs22^\times
&\pb001010
&1&&&\times (q^2-1)(q^2-q)
\\
\sbs23^\times
&\pb01010*
&1& \frac{q-1}2 & \frac{q-1}2&\times q^2(q-1)^2\\
\sbs33^\times
&\pb01010*
&1& \frac{q-1}2 & \frac{q-1}2 &\times q(q^2-1)(q^2-q)\\
\hline
\end{array}
\]

As a conclusion, we have the following counts:
\[
\scriptsize
\begin{array}{r||ccccccc}
\hline
\text{Subspace}
&\co_0&\odr&\ods&\odi&\ocs&\ots&\oti\\
\hline
\sbs12
&1&q^2-1&q(q-1)&&q(q-1)^2\\
\sbs13
&1&2q^2-q-1&\frac12q(q^2-1)&\frac12q(q-1)^2&q(q-1)^2&(q-1)^2q^2
\\
\sbs22
&1&q^2-1&q(q^2-1)&&(q^2-1)(q^2-q)&&\\
\sbs23
&1&2q^2-q-1&\frac12q(3q^2-2q-1)&\frac12q(q-1)^2&(q-1)^2(2q^2+q)&\frac{1}{2} q^2 (q - 1)^2(q + 1)&\frac12q^2(q-1)^3\\
\sbs33
&1&(q-1)(q+1)^2&\frac12q(q-1)(q+1)^2&\frac12q(q-1)^2(q+1)&q(q^2-1)^2&\frac12(q^3-q)^2&\frac12(q^2-q)^2(q^2-1)\\
\hline
\end{array}
\]

We now deduce $M$.
The vectors $(W\cap\co_i)_i\in\R^7$
for
\[
W=\sbs 00,\sbs11,\sbs02,\sbs03,\sbs13,\sbs22,\sbs33
\]
span $\R^7$. We have:
\begin{theorem}\label{thm:2sym22-A}
Suppose $\chr(\F_q)\neq2$.
The matrix $q^6M$ is given by
\[
\scriptsize
\left[ \begin {array}{ccccccc}1& ( q-1 )  ( q+1
 ) ^{2}&( q-1 )  ( q+1 ) ^{2}q/2&
 ( q-1 ) ^{2}q ( q+1 )/2 & ( q+1 ) ^{2}
q ( q-1 ) ^{2}&( q+1 ) ^{2}{q}^{2} ( 
q-1 ) ^{2}/2&( q-1 ) ^{3}{q}^{2} ( q+1
 )/2 \\ \noalign{\medskip}1&{q}^{2}-q-1&q ( 2\,q+1
 )  ( q-1 )/2 &-q ( q-1 )/2 & ( {q}^
{2}-q-1 ) q ( q-1 ) &-{q}^{2} ( q-1
 )  ( q+1 )/2 &-{q}^{2} ( q-1 ) ^{2}/2
\\ \noalign{\medskip}1& ( 2\,q+1 )  ( q-1 ) &
\,q ( {q}^{2}-2\,q-1 )/2 &q ( q-1 ) ^{2}/2&-
 ( q-1 )  ( q+1 ) q&{q}^{2} ( q-1
 ) ^{2}/2&-{q}^{2} ( q-1 ) ^{2}/2
\\ \noalign{\medskip}1&-q-1&( q-1 )  ( q+1
 ) q/2&( {q}^{2}+1 ) q/2&- ( q-1 ) 
 ( q+1 ) q&-{q}^{2} ( q-1 )  ( q+1
 )/2 &{q}^{2} ( q-1 )  ( q+1 )/2 
\\ \noalign{\medskip}1&{q}^{2}-q-1&-( q+1 ) q/2&-q
 ( q-1 )/2 &q&-{q}^{2} ( q-1 )/2 &{q}^{2}
 ( q-1 )/2 \\ \noalign{\medskip}1&-q-1&q ( q-1
 )/2 &-q ( q-1 )/2 &-q ( q-1 ) &{q}^{2}&0
\\ \noalign{\medskip}1&-q-1&-( q+1 ) q/2&( q
+1 ) q/2& ( q+1 ) q&0&-{q}^{2}\end {array} \right].
\]
\end{theorem}

\section{$2\otimes\Sym^2(3)$}\label{sec:pairs_ternary}

We come now to the prehomogeneous vector space
$(\GL_2 \times \GL_3, 2\otimes\Sym^2(3))$, 
which we call the `quartic case' because (roughly speaking) 
it parametrizes quartic fields \cite{WY} and rings \cite{HCL3}.
We assume that $\chr(\F_q)\neq 2$ in the argument,
in which case there are $20$ orbits over $\F_q$.

We write an element of $V=\F_q^2\otimes\sym^2(\F_q^3)$ as follows:
\[
x=(A,B)=
\begin{bmatrix}a&b&c&d&e&f\\ a'&b'&c'&d'&e'&f'\end{bmatrix},
\]
where $A,B$ are
the ternary quadratic forms
\[
A=au^2+buv+cuw+dv^2+evw+fw^2,
\quad
B=a'u^2+b'uv+c'uw+d'v^2+e'vw+f'w^2.
\]
We also regard $A$ and $B$ as the three-by-three symmetric matrices
representing the forms.
Let $G_1=\gl_2(\F_q)$, $G_2=\gl_3(\F_q)$ and $G:=G_1\times G_2$.
We consider the group action of $G$ given by
\[
\left(\twtw \alpha\beta\gamma\delta, g_2\right) \circ (A, B)
=(\alpha g_2 A g_2^T + \beta g_2 B g_2^T, \gamma g_2 A g_2^T + \delta g_2 B g_2^T).
\]

The {\em cubic resolvent} $r_x\in\sym^3(\F_q^2)$
and the discriminant $\Disc(x)$ of $x\in V$ are respectively defined by
\begin{align*}
r_x(u,v)&:=-4\det(Au+Bv)\in\sym^3(\F_q^2),\\
\Disc(x)&:=\disc(r_x),
\end{align*}
and as before we say $x$ is singluar if $\Disc(x) = 0$.
Then $r_{(g_1,g_2)\circ x}=\det g_1(\det g_2)^2(g_1\circ r_x)$,
and thus $\Disc(g\circ x)=(\det g_1)^6(\det g_2)^8\Disc(x)$.

Writing the symmetric bilinear form \eqref{eq:sym23-bilinear-form}
on $\sym^2(\F_q^3)$ by $[A,A']'$, we define
a symmetric bilinear form on $V$ by
\begin{equation}\label{eq:2sym23-bilinear-form}
[(A, B), (A', B')] = [A, A']' + [B, B']'.
\end{equation}
Then we have $[gx,g^{-T}x']=[x,x']$ and so
$(G,V)$ satisfies Assumption \ref{assmp:dual-identify}.

\subsection{Orbit description}

The aim of this subsection is to give an orbit description over $\F_q$.
Note first that the non-singular orbits are known by
Wright-Yukie \cite{WY},
where they showed that the set of non-singular orbits
corresponds bijectively to the set of isomorphism classes
of etale quartic algebras of $\F_q$. Hence there are five of them.
Moreover, they gave a natural geometric interpretation of this result.
An $x=(A,B)$ determines two conics in $\P^2(\overline \F_q)$,
and $x$ is non-singular if and only if they are of complete intersection.
Thus to a non-singular $x$,
we attach one of the symbols
$(1111)$, $(112)$, $(22)$, $(13)$ or $(4)$,
identifying the degrees of the residue fields at the points of intersections.
It is clear that elements in a non-singular orbit
posses the same symbol, and they actually showed that
elements having the same symbol lie in the same orbit.
This gives a satisfactory description of the non-singular orbits.
We denote these orbits respectively by
$\co_{1111}$,
$\co_{112}$,
$\co_{22}$,
$\co_{13}$
and
$\co_{4}$.

Bhargava \cite[Lemma 21]{HCL3} described
many of the singular orbits and computed their cardinalities,
and we will build upon his work as well.

To classify singular orbits, it is useful
to think of certain ``higher singular'' conditions of
$x=(A,B)$,
described as follows:
\begin{enumerate}
\item[(D)]
$A$ and $B$ are linearly dependent;
\item[(C)]
$A$ and $B$ share a common linear factor;
\item[(B)]
After a change of variables, $A$ and $B$ can be written as
a pair of binary quadratic forms (each in the same two variables).
\end{enumerate}
In other words, the $G$-orbit of $x$ contains an element
in the subspace
\[
W_{\rm (D)}=
\begin{bmatrix}
	0&0&0&0&0&0\\
	*&*&*&*&*&*\\
\end{bmatrix},
\quad
W_{\rm (C)}=
\begin{bmatrix}
	0&0&*&0&*&*\\
	0&0&*&0&*&*\\
\end{bmatrix},
\quad
W_{\rm (B)}=
\begin{bmatrix}
	0&0&0&*&*&*\\
	0&0&0&*&*&*\\
\end{bmatrix},
\]
respectively.
Geometrically,
$x$ is of type (D) if $A=0$ and $B=0$ are identical conics
so they are doubled;
$x$ is of type (C) if $A=0$ and $B=0$ are both reducible conics
sharing a common line;
$x$ is of type (B) if $A=0$ and $B=0$ are both reducible as conics over $\overline \F_q$ and pass through a single common point.
(These conditions are not mutually exclusive, and should be interpreted
a bit loosely; for example the doubled zero conic belongs to all three subspaces.)

We will prove that if $x$ is singular
but not of type (D), (C) or (B),
then two conics $A=0$ and $B=0$
intersect with each other in exactly four points
counting multiplicities, and that the state
of intersection is a complete invariant for those orbits.
There are six types, and we attach symbols
$(1^4)$, $(1^31)$, $(1^21^2)$,
$(2^2)$, $(1^211)$ or $(1^22)$.
Hence we denote these respective six orbits by
$\co_{1^4}$,
$\co_{1^31}$,
$\co_{1^21^2}$,
$\co_{2^2}$,
$\co_{1^211}$
and
$\co_{1^22}$,
after we have the assertion.

We prove:
\begin{lemma}\label{lem:singular-orbit}
Singular elements are either of type {\rm (D)}, {\rm (C)} or {\rm (B)},
or $G$-equivalent to one of the six elements below.
Here, $l\in\F_q^\times$ is a non-square element.
\[
(w^2,uw+v^2),
(vw,uw+v^2),
(w^2,u^2-v^2),
(w^2,u^2-lv^2),
(v^2-w^2,uw),
(v^2-lw^2,uw).
\]
The symbols of these six elements are
$(1^4)$, $(1^31)$, $(1^21^2)$,
$(2^2)$, $(1^211)$ or $(1^22)$,
respectively.
\end{lemma}
%\begin{remark}
%For the six elements above, two conics $A=0$ and $B=0$
%intersect with each other in exactly four points
%(counting multiplicities), and the symbols are
%$1^4,1^31,1^21^2,2^2,1^211,1^22$, respectively.
%\end{remark}
\begin{proof}
Suppose $x=(A,B)$ is singular.
Then since $r_x$ is a singular binary cubic form,
we may assume that the coefficients of $u^3$ and $u^2v$ of
$r_x$ are both zero.
Thus in particular $\det(A)=0$ 
and hence $\rank(A)\leq2$.
\begin{enumerate}
\item
If $\rank(A)=0$, then $x$ is of type (D).
\item
Let $\rank(A)=1$. Then by $\gl_3$, we may assume that $A=w^2$. 
Let $B=au^2+buv+cuw+dv^2+evw+fw^2$.
We look at $B(u,v,0)=au^2+buv+dv^2\in\sym^2(\F_q^2)$.
\begin{enumerate}
\item
If $au^2+buv+dv^2$ is a zero form,
then $x$ is of type (C).
\item
Suppose $au^2+buv+dv^2$ is non-zero but singular.
By a linear change of $u$ and $v$,
we may assume $a=b=0$ and $d=1$, and thus
$(A,B)=(w^2,v^2+w(cu+ev+fw))$.
\begin{enumerate}
\item
If $c=0$, then $x$ is of type (B).
\item
If $c\neq0$, we may replace $cu+ev+fw$ with $u$ via $\gl_3$,
and thus $(A,B)=(w^2,v^2+uw)$.
\end{enumerate}
\item
Suppose $au^2+buv+dv^2$ is non-singular.
By a linear change of $u$ and $v$, we may assume $b=0$,
and hence $ad\neq0$.
Then
\[
B=au^2+cuw+dv^2+evw+fw^2=a(u+\tfrac{c}{2a}w)^2+d(v+\tfrac{e}{2d}w)^2+*w^2.
\]
Replacing $u+\tfrac{c}{2a}w$ with $u$ and $v+\tfrac{e}{2d}w$ with $v$
via $\gl_3$, and further eliminating the $w^2$-term using $A=w^2$,
we have $x=(w^2,au^2+dv^2)$. Since $ad\neq0$, this is equivalent to
one of the middle two in the list.
\end{enumerate}

\item
Let $\rank(A)=2$.
Using the $\gl_3$ action,
Proposition \ref{prop:orbit_ternary} allows us to assume that
$A=av^2+bvw+cw^2$.
Since the $u^2v$ term of $r_x$ vanishes,
$B$ is of the form
$B=duv+euw+fv^2+gvw+hw^2$. 
\begin{enumerate}
\item
If $d=e=0$, then $x$ is of type (B).
\item
If $(d,e)\neq(0,0)$, then by a linear change of $v$ and $w$,
we may assume that $(d,e)=(0,1)$.
\begin{enumerate}
\item
Suppose $a=0$. Then $A=w(bv+cw)$.
Since $\rank(A)=2$, we have $b\neq0$ and hence
may replace $bv+cw$ with $v$ via $\gl_3$.
So $A=vw$. 
On the other hand, since $B=fv^2+w(u+gv+hw)$,
we may replace $u+gv+hw$ with $u$ and thus $B=fv^2+uw$.
\begin{enumerate}
\item
If $f=0$, then $x=(vw,uw)$ is of type (C).
\item
If $f\neq0$, then replace $u$ with $fu$ and we have
$B=f(v^2+uw)\sim v^2+uw$.
\end{enumerate}
\item
Suppose $a\neq0$.
We may assume $a=1$. We can eliminate $b$ and $f$
and thus may assume $(A,B)=(v^2+cw^2,w(u+gv+hw))$.
This lies in the orbit of $(v^2+cw^2,uw)$
since we can replace $u+gv+hw$ with $u$ via $\gl_3$.
This is equivalent to the one of the last two in the list.
\end{enumerate}
\end{enumerate}
\end{enumerate}
This finishes the proof.
\end{proof}

We now give our orbit description, extending \cite[Lemma 21]{HCL3}.
To describe the orbit sizes, we write
\[
s(a,b,c,d):=(q-1)^aq^b(q+1)^c(q^2+q+1)^{d/2},
\]
where $d$ is always even.
Note that its degree as a polynomial in $q$ is $a+b+c+d$.

\begin{proposition}\label{prop:orbit-2sym23}
Assume $\chr(\F_q)\neq 2$.
The action of $G$ on $V$ has twenty orbits,
of which fifteen are singular.
For each orbit, 
the table below lists an orbital representative,
the type of the resolvent,
and the number of rational common zeros in $\P^2(\F_q)$,
of any element in the orbit,
and the size of the orbit:
\[
\begin{array}{llccl}
\text{\rm Orbit}
&\text{\rm Representative}
&\text{\rm Resolvent}
&\text{\rm Common zeros}
&\text{\rm Orbit size}
\\
\hline
\co_0
&(0,0)
%	&\F_q[x,y,z]/(x,y,z)^2
&(0)
&q^2+q+1
		&1
\\
\odr
&(0,w^2)
%	&K[x,y,z]/(x,y,z)^2
&(0)
&q+1
		&s(1,0,1,2)
\\
\ods
&(0,vw)
%	&K[x,y,z]/(x,y,z)^2
&(0)
&2q+1
		&s(1,1,2,2)/2
\\
\odi
&(0,v^2-lw^2)
%	&K[x,y,z]/(x,y,z)^2
&(0)
&1
		&s(2,1,1,2)/2
\\
\odg&(0,u^2-vw)
%	&K[x,y,z]/(x,y,z)^2
&(1^3)
&q+1
		&s(2,2,1,2)
\\
\ocs&
(w^2,vw)
%	&K[x,y]/(x^2,xy+y^2)
&(0)
&q+1
		&s(2,1,2,2)
\\
\ocg&
(vw,uw)
%	&K\times K[x,y]/(x,y)^2
&(0)
&q+2
		&s(2,3,1,2)
\\
\ots
&(w^2,v^2)
%	&K[x,y]/(x^2,y^2)
&(0)
&1
		&s(2,2,2,2)/2
\\
\oti
&(vw,v^2+lw^2)
%	&K[x,y]/(x^3,xy-ax^2,y^2-bx^2)
&(0)
&1
		&s(3,2,1,2)/2
\\
\co_{1^4}
&(w^2,uw+v^2)
%	&K[x]/(x^4)
&(1^3)
&1
		&s(3,2,2,2)
\\
\co_{1^31}
&(vw,uw+v^2)
%	&K[x]/(x^3)\times K
&(1^3)
&2
		&s(3,3,2,2)
\\
\co_{1^21^2}
&(w^2,u^2-v^2)
%	&K[x]/(x^2)\times K[y]/(y^2)
&(1^21)
&2
		&s(2,4,2,2)/2
\\
\co_{2^2}&
(w^2,u^2-lv^2)
%	&K_2[x]/(x^2)
&(1^21)
&0
		&s(3,4,1,2)/2
\\
\co_{1^211}
&(v^2-w^2,uw)
%	&K[x]/(x^2)\times K^2
&(1^21)
&3
		&s(3,4,2,2)/2
\\
\co_{1^22}
&(v^2-lw^2,uw)
%	&K[x]/(x^2)\times K_2
&(1^21)
&1
		&s(3,4,2,2)/2
\\
\co_{1111}
&(uw-vw,uv-vw)
%	&K^4
&(111)
&4
		&s(4,4,2,2)/24
\\
\co_{112}
&(vw,u^2-v^2-lw^2)
%	&K^2\times K_2
&(12)
&2
		&s(4,4,2,2)/4
\\
\co_{22}
&(vw,u^2-lv^2-lw^2)
%	&K_2\times K_2'
&(111)
&0
		&s(4,4,2,2)/8
\\
\co_{13}&
(uw-v^2,B_3)
%	&K\times K_3
&(3)
&1
		&s(4,4,2,2)/3
\\
\co_4
&(uw-v^2,B_4)
%	&K_4
&(12)
&0
		&s(4,4,2,2)/4
\\
\end{array}
\]
Here $B_3$ and $B_4$ in the last two rows are
$B_3=uv+a_3v^2+b_3vw+c_3w^2$ and
$B_4=u^2+a_4uv+b_4v^2+c_4vw+d_4w^2$,
where
$X^3+a_3X^2+b_3X+c_3$ and $X^4+a_4X^3+b_4X^2+c_4X+d_4$
are respectively irreducible cubic and quartic polynomials over $\F_q$
(and as before $l\in\F_q^\times$ is a non-square element).
\end{proposition}

\begin{proof}
If $x$ is of type (D), then by the orbit description of $\sym^2(3)$,
$x$ is equivalent to one of the first five elements in the table.
If $x$ is of type (B), then by the orbit description of
$2\otimes\sym^2(2)$,
$x$ is equivalent to either one of the first four elements in the table,
or to $(w^2,vw)$, $(w^2,v^2)$ or $(vw,v^2+lw^2)$.
If $x$ is of type (C), then
$x$ is equivalent to either
$(0,0)$, $(0,w^2)$, $(0,vw)$, $(w^2,vw)$ or $(vw,uw)$.
Hence by Lemma \ref{lem:singular-orbit}
and the result of \cite{WY} mentioned above,
any element in $V$ is equivalent to one of the twenty elements
in the table.

We confirm that their orbits are all different.
This is immediate,
except for the possibility that $(w^2,v^2)$
and $(vw,v^2+lw^2)$ may lie the same orbit,
by comparing
%We note that almost all of the twenty elements above
%are distinguished by comparing 
the three invariants
rank, resolvent, and the number of common zeros in $\mathbb P^2$.
We show that
$(w^2,v^2)\sim(vw,v^2-w^2)$
and $(vw,u^2+lw^2)$ are not in the same orbit.

We embed $W=2\otimes\sym^2(2)$ into $V=2\otimes\sym^2(3)$
via $\pb abcdef\mapsto \pt 000abc000def$, and regard it as a subspace.
Let $y=\pb 010def=\pt 000010000def$ and suppose
$g\circ y\in W$ for a $g=(g_1,g_2)\in G$.
Since $W$ is invariant under $G_1$,
we have $(1,g_2)\circ y\in W$ as well.
Let
$g_2=
\left[\begin{smallmatrix}
	h&\alpha&\beta\\
	i&j&k\\
	l&m&n
\end{smallmatrix}\right]$.
We claim that $\alpha=\beta=0$.
The first row of $y$ is the ternary quadratic form $vw$,
and it is transformed by $g_2$ to $(\alpha u+jv+mw)(\beta u+kv+nw)$.
Since this form involves only the variables $v$ and $w$,
we have $\alpha\beta=\alpha k+\beta j=\alpha n+\beta m=0$.
If $\alpha\neq0$, then $\beta=k=n=0$ and so $g_2$ is not invertible.
Hence $\alpha=0$. Similarly, we have $\beta=0$.

%Let $g_2=\left[\begin{smallmatrix}a&b&c\\g&h&i\\j&k&l\end{smallmatrix}\right]$.
%(\orange{This $a$, $b$, and $c$ are unrelated to the previous ones, correct? I suggest we use different letters, to avoid
%potential confusion.})
%We claim that $b=c=0$.
%The first row of $y$ is the ternary quadratic form $vw$,
%and it is transformed by $g_2$ to $(bu+hv+kw)(cu+iv+lw)$.
%Since this form involves only the variables $v$ and $w$,
%we have $bc=0, bi+ch=0, bl+ck=0$.
%If $b\neq0$, then $c=i=l=0$ and so $g_2$ is not invertible.
%Hence $b=0$. Similarly, we have $c=0$.
This shows that elements of the form $\pb 010def$
are $G$-equivalent in $V$ if and only if they are $\gl_2\times\gl_2$-equivalent
in $W$, and thus the difference of the orbits is asserted.

We count the orbit sizes.
Note that the argument above also gives
an expression of the stabilizers of
$\pt 000010000def$ in $G$ in terms of 
the stabilizers of $\pb 010def$ in $\gl_2\times\gl_2$.
Namely, if $g=(g_1,g_2)\in G$ stabilizes $y=\pt 000010000def$,
then $g_2$ must be of the form 
$\left[\begin{smallmatrix}
	h&0&0\\
	i&j&k\\
	l&m&n
\end{smallmatrix}\right]$,
and for such a $g=(g_1,g_2)$, we immediately find that
$g$ stabilizes $y$ if and only if
$(g_1,\twtw jkmn)\in\gl_2\times\gl_2$ stabilizes $\pb 010def$.
Therefore, the orbit sizes of $\ocs$, $\ots$ and $\oti$ in $V$
are multiplied by $q^2+q+1$ to those of in $W$.
Each of the orbits $\odr$, $\ods$, $\odi$ and $\odg$
consisting of doubled conics is in bijection with a pair
$(\co, \gamma)$ where $\co$ is a $\gl_1\times\gl_3$ orbit in
$\sym^2(3)$ and $\gamma \in \mathbb P^1$,
so that the orbit sizes are deduced from those
for $\sym^2(3)$.
To compute the orbit size of $\ocg$, we determine the
stabilizers of $x=(vw,uw)$.
Suppose $g=(g_2,g_3)$ stabilizes $x$.
Let $g_3$ translates $u,v,w$ to $l_1,l_2,l_3$
respectively; these are independent ternary linear forms.
Also let $g_2^{-1}=\twtw \alpha\beta\gamma\delta$.
Then $gx=x$ means
$(l_2l_3,l_1l_3)=((\alpha v+\beta u)w,(\gamma v+\delta u)w)$.
Thus $l_3$ must coincide with $w$ up to scaling,
and therefore $l_1,l_2$ are linear forms in $u$ and $v$.
Now it is easy to see that
$\Stab(x)=\left\{\left(\twtw acbd^{-1},
\left[\begin{smallmatrix}a&b&\\c&d&\\&&e\end{smallmatrix}\right]
\right)\right\}\cong\gl_1\times\gl_2$;
we conclude $|\ocg|=|G|/|\Stab(x)|=|\gl_3|/|\gl_1|$.

The orbit sizes
of the latter eleven orbits are determined in \cite[Lemma 21]{HCL3}.

Finally, we compute the resolvents. Except for the last one,
this follows by rather easy case by case computation.
For $x=(vw-u^2,B_4)$, we find that
\[
r_x(u, 1) = u^3 - b_4 u^2 + (a_4 c_4 - 4d_4) u - (a_4^2 d_4 - 4b_4 d_4 + c_4^2)
\]
is the cubic resolvent of the polynomial
$Q(u) = u^4 - a_4u^3 + b_4u^2 - c_4 u + d_4$ and thus $\disc(r_x) = \disc(Q)$.
Since
$\disc(r_x) = \disc(Q)\equiv \disc(\F_{q^4}/\F_q)\pmod{(\F_q^\times)^2}$
is not a square in $\F_q$,
$r_x\in\sym^3(\F_q^2)$ is of type $(12)$.
This finishes the proof.
\end{proof}

\subsection{Counting elements in subspaces}
We now demonstrate our counts.
For $i,j\in\{0,1,2,3,4,5,6\}$, let $\sbs ij$
be the subspace consisting of pairs of forms $(A, B)$ such that
the last $i$ entries of $A$ and the last $j$ entries of $B$
are arbitrary, and other entries are $0$.
We largely follow the method in the previous section.
We define $\sbs ij^\times$ for $i<j$ and $i=j$
in the same way;
then \eqref{eq:wij} and \eqref{eq:wii}
remains true, and we argue inductively.
Note also that the counts of
$\sbs ij$ for $i,j\leq 3$
are obtained in the previous section,
while the counts of $\sbs 0j$ for $4\leq j\leq6$
are obtained in Section \ref{sec:sym23}.
For $X\subset V$, we use the same symbol $X$
to denote $(X\cap\co_i)_i\in\R^{20}$.

\medskip
\noindent
{\bf (I)}
%\subsubsection*{(I)}
We first count $\sbs 14, \sbs 15, \sbs 16, \sbs 24, \sbs 25, \sbs26$
and $\sbs 44$ by the same method as in the previous section.
The following diagram illustrates our induction process:
\[
\xymatrix@R=5mm@C=7mm{
 \sbs 23\ar[r]
&\sbs 24\ar[r]
&\sbs 25\ar[r]
&\sbs 26
\\
 \sbs 13\ar[ur]\ar[r]
&\sbs 14\ar[ur]\ar[r]\ar[u]
&\sbs 15\ar[ur]\ar[r]\ar[u]
&\sbs 16\ar[u]
\\
 \sbs 03\ar[ur]
&\sbs 04\ar[ur]\ar[u]
&\sbs 05\ar[ur]\ar[u]
&\sbs 06\ar[u]
}
\qquad
\xymatrix@R=5mm@C=7mm{
\\
&\sbs 44
\\
\sbs 22\ar[ur]
&\sbs 24\ar[u]
}
\]
We have previously counted
$\sbs 03, \sbs13, \sbs23, \sbs04, \sbs05,\sbs06, \sbs22$,
and now analyze $\sbs ij^\times$ for the seven subspaces mentioned above.
The following two tables give the summary of our counts:
\begin{align*}
&
\begin{array}{r||c|cccccccc|c}
\hline
\text{Subset}
&Y
&\ocs&\ocg&\ots&\co_{1^4}&\co_{1^31}&\co_{1^21^2}&\co_{2^2}&\co_{1^211}
&\text{multiplier}
\\
\hline
\sbs14^\times
&\pt000001001*00
&1&&&q-1&&&&&\times q^2(q-1)^2
\\
\sbs15^\times
&\pt000001010000
&&&&&&1&&&\times q^4(q-1)^2
\\
\sbs16^\times
&\pt000001100**0
&&&1&q-1&&\frac{q(q-1)}2&\frac{q(q-1)}2&&\times q^3(q-1)^2
\\
\sbs24^\times
&\pt000010001*00
&&1&&&q-1&&&&\times q^3(q-1)^2
\\
\sbs25^\times
&\pt00001001*00*
&&1&&&q-1&&&q(q-1)&\times q^3(q-1)^2
\\
\hline
\end{array}
\\
&
\begin{array}{r||c|cccccc|c}
\hline
\text{Subset}
&Y
&\co_{1^21^2}&\co_{1^211}&\co_{1^22}&\co_{1111}&\co_{112}&\co_{22}
&\text{multiplier}
\\
\hline
\sbs26^\times
&\pt000010100*0*
&1& q-1 & q-1 & \frac{(q-1)^2}4 & \frac{(q-1)^2}2 & \frac{(q-1)^2}4 & \times q^4(q-1)^2
\\
\sbs44^\times
&\pt00010*001000
&1& \frac{q-1}2 & \frac{q-1}2 &&&& \times q^3(q^2-1)(q^2-q)
\\
\hline
\end{array}
\end{align*}
As with $2 \otimes \Sym^2(2)$,
the counts for each $W^{\times}$ are equal to the relevant 
multiplier times those for the associated $Y$. These reductions are obtained
in a very similar manner to those for $2 \otimes \Sym^2(2)$ and so we limit ourselves to an outline of the necessary steps in the 
more difficult cases.
\begin{itemize}
\item
$\sby 14\ni\pt 000001001a00=(w^2,uw+av^2)$
is in $\ocs$ if $a=0$; otherwise,
$(w^2,uw+av^2) \sim(w^2,uw+v^2)\in\co_{1^4}$, as was listed in our table
of orbital representatives.
\item
$\sby 15\ni\pt 000001010000=(w^2,uv)$
is in $\co_{1^21^2}$. (The reduction, slightly more complicated
than previous ones, is most easily verified in the order
$
\pt 00000101**** = q\pt 000001010***=q^3 \pt 00000101000* = q^4 \pt 000001010000.
$) 

\item
$\sby 16\ni\pt 000001100{-a}b0=(w^2,u^2-av^2+bvw)$
is in $\ots$ if $a=b=0$
and is in $\co_{1^4}$ if $a=0$ but $b\neq0$.
If $a\neq0$, a routine intersection multiplicity computation establishes that
it is either in $\co_{1^21^2}$ or in $\co_{2^2}$
according as $a$ is a quadratic residue or non-residue.
\item
$\sby 24\ni\pt 000010001a00=(vw,uw+av^2)$
is in $\ocg$ if $a=0$ and
is in $\co_{1^31}$ elsewhere.
(The reduction may be verified in the order
$
\pt 00001*001*** = q\pt 000010001***=q^3 \pt 000010001*00
$.)
\item
$\sby 25\ni\pt 00001001a00b=(vw,uv+w(au+bw))$
is in $\ocg$ if $a=b=0$
and is in $\co_{1^31}$ if $a=0$ but $b\neq0$.
If $a\neq0$, we see that it has three rational zeros
$(1:0:0)$, $(0:1:0)$ and $(b:0:-a)$,
and thus is in $\co_{1^211}$.
($(1:0:0)$ is the double zero.)
\item
$\sby 26\ni\pt 000010100{-a}0{-b}=(vw,u^2-av^2-bw^2)$
has four common zeros $(\pm\sqrt{a}:1:0),(\pm\sqrt{b}:0:1)$.
The multiplicity and rationality are determined by
whether $a$ and $b$ are respectively zero,
a quadratic residue, or a quadratic non-residue,
and we have the counts.
(The reduction may be verified in the order
$
\pt 00001*1***** = 
q \pt 0000101***** = 
q^3 \pt 000010100*** = 
q^4 \pt 000010100*0*$.)
\item
$\sby 44\ni \pt00010{-a}001000=(v^2-aw^2,uw)$
is in $\co_{1^21^2}$, $\co_{1^211}$ or $\co_{1^22}$ according as
$a$ is zero, a quadratic residue, or a quadratic non-residue.
\end{itemize}

\medskip
\noindent
{\bf (II)}
%\subsubsection*{(II)}
Secondly, again by the same method, we count for
the following five subspaces
\[
\begin{array}{l}
W_1:=\pt00000000*0**,\\
W_2:=\pt00000*00*0**,\\
W_3:=\pt00*0**00*0**
\end{array}
\qquad
\text{and}
\qquad
\begin{array}{l}
W_4:=\pt000000000*0*,\\
W_5:=\pt000*0*000*0*
\end{array}
\]
with the following steps:
\[
\xymatrix@R=5mm@C=7mm{
\sbs12\ar[r]
&W_2
\\
\sbs 02\ar[ur]
&W_1\ar[u]
}
\qquad
\xymatrix@R=5mm@C=7mm{
&W_3
\\
\sbs 11\ar[ur]
&W_2\ar[u]
}
\qquad
\xymatrix@R=5mm@C=7mm{
&W_5
\\
\sbs 00\ar[ur]
&W_4\ar[u]
}
\]
Our result is:
\begin{align*}
&\begin{array}{r||ccccc}
\hline
\text{Subspace}
&\co_0&\odr&\ods&\ocs&\ocg\\
\hline
W_1=\pt00000000*0**&1&q-1&q(q^2-1)\\
%\pt00000*0000**&1&q^2-1&q(q-1)&q(q-1)^2\\
W_2=\pt00000*00*0**&1&q^2-1&q(q^2-1)&(q^2-1)(q^2-q)\\
W_3=\pt00*0**00*0**&1&q^2-1&q(q^2-1)(q+1)&q(q^2-1)^2&q^2(q^2-1)(q^2-q)\\
\hline
\end{array}\\
&\begin{array}{r||ccccc}
\hline
\text{Subspace}
&\co_0&\odr&\ods&\odi&\ots\\
\hline
W_4=\pt000000000*0*&1&2(q-1)&\frac12(q-1)^2&\frac12(q-1)^2\\
W_5=\pt000*0*000*0*&1&2(q^2-1)&\frac12(q-1)(q^2-1)&\frac12(q-1)(q^2-1)&(q^2-1)(q^2-q)\\
\hline
\end{array}
\end{align*}

The counts for $W_1$ and $W_4$ are immediate.
To work as before,
we define $W_2^\times$, $W_3^\times$ and $W_5^\times$ in the same way:
\[
W_2^\times:=\left\{\pt00000a00b0cd\ \vrule\ a,b\neq0\right\},
\]
and
\[
W_3^\times:=\left\{\pt00a0bc00d0ef\ \vrule\ ae-bd\neq0\right\},
\quad
W_5^\times:=\left\{\pt000a0b000c0d\ \vrule\ ad-bc\neq0\right\}.
\]
Then in this case we immediately see that
$W_2^\times\subset\ocs$,
$W_3^\times\subset\ocg$ and
$W_5^\times\subset\ots$, and thus we have the table.

\medskip
\noindent
{\bf (III)}
%\subsubsection*{(III)}
Thirdly, we count for
\[
\sbs 55=\pt0*****0*****
\qquad
\text{and}
\qquad
W_6:=\pt0*0***0*0***.
\]
For these subspaces, we are counting the number of elements
of each $\co_i$ with, respectively,
having $[0:0:1]$ as a common zero,
and having $[0:1:0]$ and $[0:0:1]$ as common zeros.
Let $n_i$ be the number of common zeros of any $x\in\co_i$.
Then we have
\begin{align*}
|\sbs 55\cap\co_i|&=\frac{n_i}{p^2+p+1}|\co_i|,\\
|W_6\cap\co_i|&=\frac{n_i(n_i-1)}{(p^2+p+1)(p^2+p)}|\co_i|.
\end{align*}

\medskip
\noindent
{\bf (IV)}
%\subsubsection*{(IV)}
Finally, we study
\[
W_7:=\pt 000***\ast**000.
\]
We work directly for this case.
Let $x=\pt 000abcdef000=(av^2+bvw+cw^2,u(du+ev+fw))$.
\begin{itemize}
\item
Let $e=f=0$. Then according as
$av^2+bvw+cw^2\in\sym^2(2)$ is of type $(0),(1^2),(11),(2)$,
$x$ is
in $\co_0$, $\odr$, $\ods$, $\odi$ if $d=0$ and
in $\odr$, $\ots$, $\co_{1^21^2}$, $\co_{2^2}$ if $d\neq0$.
\item
Suppose $(e,f)\neq(0,0)$.
The subset consisting of such $x$
has a ($q^2-1$)-to-$1$ map to $Y=\{\pt 000abcd01000\}$.
%by the linear change of variables in $v$ and $w$.
We assume $x=(av^2+bvw+cw^2,u(du+w))\in Y$.
\begin{itemize}
\item
Let $a=0$.
Then according as $b=c=0$, $b=0$ but $c\neq0$, $b\neq0$,
$x$ is
in $\ods$, $\ocs$, $\ocg$ if $d=0$ and
in $\ods$, $\ots$, $\co_{1^211}$ if $d\neq0$.
\item
Let $a\neq0$.
Then according as $av^2+bvw+cw^2$ is of type $(1^2), (11), (2)$,
$x$ is
in $\co_{1^21^2}$, $\co_{1^211}$, $\co_{1^22}$ if $d=0$ and
in $\co_{1^21^2}$, $\co_{1111}$, $\co_{22}$ if $d\neq0$.
\end{itemize}
\end{itemize}
As a result, the number of elements of $W_7$ in
the twenty orbits are respectively counted as
\begin{gather*}
\textstyle
1,
q^2+q-2,
\frac32(q^3-q),
\frac12q(q-1)^2,
0,
(q-1)(q^2-1),
(q^2-q)(q^2-1),
(q^2-q)(q^2-1),
0,\\
\textstyle
0,
0,
\frac12(q^3-q)(2q^2-q-1),
\frac12q(q-1)^3,
\frac32(q^3-q)(q-1)^2,
\frac12(q^3-q)(q-1)^2,
\\
\textstyle
\frac12(q^3-q)(q-1)^3,
0,
\frac12(q^3-q)(q-1)^3,
0,
0,
\end{gather*}
where the ordering of the orbits are as in Proposition \ref{prop:orbit-2sym23}.
\begin{proposition}
The vectors $(|W\cap\co_i|)_i$ for
\begin{align*}
W=
&
\sbs00,\sbs66;
\sbs11,\sbs55;
\sbs04,\sbs26;
\sbs05,\sbs16;
\sbs14,\sbs25;\\
&
\sbs22,\sbs44;
\sbs33,W_3;
W_5,W_6;
\sbs06;
\sbs15;
\sbs24;
W_7
\end{align*}
span $\R^{20}$.
\end{proposition}
\begin{proof}
With the twenty vectors ordered as in the table below,
one checks one column at a time
that the span of the vectors $(|W\cap\co_i|)_i$ for those
$W$ listed in the first $j$ columns,
includes the characteristic functions of those orbits
$\calO$ listed beneath them.
Since all twenty orbits appear in the table,
this gives the proposition.

\begin{multline*}
%\begin{align*}
\begin{array}{|l||c|c|c|c|c|c|c|c}
\hline
\text{Subspace added}
&\sbs 00
&\sbs 11
&\sbs 04,\sbs05,\sbs06
&\sbs22
&W_3
&W_5
&\sbs33
&
\\\hline
\text{Spanned}
&\co_0
&\odr
&\ods,\odi,\odg
&\ocs
&\ocg
&\ots
&\oti
&
\\\hline
\end{array}\\
\begin{array}{c|c|c|c|c|c|c|c|c|}
\hline
\sbs14
&\sbs15
&\sbs16
&\sbs24
&\sbs25
&\sbs44
&\sbs26,W_6,W_7
&\sbs55
&\sbs66=V
\\\hline
\co_{1^4}
&\co_{1^21^2}
&\co_{2^2}
&\co_{1^31}
&\co_{1^211}
&\co_{1^22}
&\co_{1111},\co_{112},\co_{22}
&\co_{13}
&\co_{4}
\\\hline
\end{array}
%\end{align*}
\end{multline*}
For those $W$ listed singly this verification is immediate.
For the two groupings of three,
we need (and can easily check) that the 3-by-3 matrices
obtained from the tables
\[
\begin{array}{r||ccccc}
\hline
&\ods&\odi&\odg\\
\hline
\sbs04
&\frac12q(3q^2-2q-1)&\frac12q(q-1)^2& q^2(q-1)^2\\
\sbs05
&\frac12q(q^2-1)(2q+1)&\frac12q(q-1)^2& q^2(q-1)(q^2-1)\\
\sbs06
&\frac12q(q+1)(q^3-1)&\frac12q(q-1)(q^3-1)& q^2(q-1)(q^3-1)\\
\hline
\end{array}
\]
and
\[
\begin{array}{l||ccc}
\hline
&\co_{1111}&\co_{112}&\co_{22}
\\
\hline
\sbs26
&\frac14q^4(q-1)^4&\frac12q^4(q-1)^4&\frac14q^4(q-1)^4
\\
W_6
&\frac{12s(4,4,2,2)}{24(p^2+p+1)(p^2+p)}
&\frac{2s(4,4,2,2)}{4(p^2+p+1)(p^2+p)}
&0
\\
W_7
&\frac12(q^3-q)(q-1)^3
&0
&\frac12(q^3-q)(q-1)^3
\\
\hline
\end{array}
\]
are invertible.
\end{proof}

\begin{theorem}\label{thm:2sym23-A}
Suppose $\chr(\F_q)\neq2$.
We have an explicit formula for $M$. (For typesetting reasons it is given on page \pageref{thm:explicit-big}, after the bibliography.)
\end{theorem}

\section*{Acknowledgments}
We would like to thank Jan Denef, Yasuhiro Ishitsuka, Kentaro Mitsui, Arul Shankar, Ari Shnidman, Nicolas Templier, and Kota Yoshioka for helpful comments.
We would especially like to thank Hiroyuki Ochiai and an anonymous referee for a careful reading and many useful comments.

This material is based upon work supported by the National Science Foundation under Grant No. DMS-1201330,
by the National Security Agency under a Young Investigator Grant,
by the JSPS KAKENHI Grant Numbers JP24654005, JP25707002,
JP16K13747, JP17H02835
and by JSPS Joint Research Project with CNRS.

\appendix
\section{Invariance of bilinear forms}\label{appendix:ibf}
In this section, we describe a general principle to construct
invariant bilinear forms and apply it to show the invariance of the
bilinear forms we introduced in this paper.

Let $V$ be a vector space over a field $K$,
with a linear action of a group $G$.
Suppose a bilinear form $[\cdot,\cdot]$ on $V$
and an involution $g\mapsto g^\iota$ on $G$
such that $[gx,g^\iota y]=[x,y]$ hold for all
$x,y\in V$ and $g\in G$ are given.
Let us consider the outer tensor representation
of $\tilde G=\gl_n(K)\times G$ on $\tilde V=K^n\otimes V$.
We regard $\tilde V$  as the space of $n$-tuples of
elements in $V$, and define a bilinear form on it by
\[
[x,y]_{\tilde V}:=\sum_i [x_i,y_i],
\qquad
x=(x_i),
y=(y_i)\in\tilde V.
\]
Then we have
\[
[gx,g^\iota y]_{\tilde V}=[x,y]_{\tilde V}
\]
for all $x,y\in \tilde V$ and $g\in \tilde G$,
where $g^\iota:=(g_1^{-T},g_2^\iota)$ for $g=(g_1,g_2)$.
To confirm this, since $(g_1,g_2)=(g_1,1)(1,g_2)$ and
the assertion for $g=(1,g_2)$ is obvious, we may assume
$g=(g_1,1)$. Let $g_1=(g_{ij})\in\gl_n(K)$ and $g_1^{-1}=(h_{ij})$.
Then
\begin{align*}
[gx,g^\iota y]_{\tilde V}
=\sum_{i,j,k}[g_{ij}x_j,h_{ki}y_k]
=\sum_{j,k}[x_j,y_k]\sum_i h_{ki}g_{ij}
=\sum_{j,k}[x_j,y_k]\delta_{k,j}
=\sum_{i}[x_i,y_i]
=[x,y]_{\tilde V},
\end{align*}
as desired.
We also note that if the bilinear form on $V$ is symmetric,
then so is the bilinear form on $\tilde V$.

We consider the space $n\otimes n=K^n\otimes K^n$
with the action of $\gl_n(K)\times\gl_n(K)$.
Let $e_1,\dots,e_n\in K^n$ be the standard basis of $K^n$.
%We put $e_{ij}:=e_i\otimes e_j\in n\otimes n$.
Then the bilinear form on $n\otimes n$,
constructed from the one dimensional
trivial representation $(\{{\rm id}\},K)$ with $[x,y]=xy$ as above,
is given by
\[
[x,y]:=\sum_{i,j}x_{ij}y_{ij}
\]
for
\[
x=\sum_{i,j}x_{ij}e_i\otimes e_j,
\quad
y=\sum_{i,j}y_{ij}e_i\otimes e_j.
\]
Hence this satisfies
\[
[gx,g^{-T}y]=[x,y],
\]
where $g^{-T}=(g_1^{-T},g_2^{-T})$
for $g=(g_1,g_2)\in\gl_n(K)\times\gl_n(K)$,
and is symmetric.

Let $\sym_2(n)$ and $\sym^2(n)$ be the symmetric subspace
and symmetric quotient of $n\otimes n$, respectively. (More specifically, $\sym_2(n)$ is the subspace of $n\otimes n$ invariant under the flip $x\otimes y\mapsto y\otimes x$, and $\sym^2(n)$ is the quotient of $n\otimes n$ by the subspace which is generated by elements of the form $x\otimes y-y\otimes x$.)
The single $\gl_n(K)$ acts on $n\otimes n$ through its diagonal embedding
$g\mapsto(g,g)$,
and this action is inherited to the actions of $\gl_n(K)$
on $\sym_2(n)$ and $\sym^2(n)$. Namely,
the linear maps
\[
\sym_2(n)
\hookrightarrow
n\otimes n
\twoheadrightarrow
\sym^2(n)
\]
are $\gl_n(K)$-equivariant.
Now assume that the characteristic of $K$ is not two.
Then the composition of the two maps is an isomorphism.
If we identify $\sym^2(n)$ with $\sym_2(n)$ via this isomorphism,
we have a bilinear form on $\sym^2(n)=\sym_2(n)$
by restricting the bilinear form
on $n\otimes n$. It is symmetric
and satisfies 
\[
[gx,g^{-T}y]=[x,y],
\]
for $x,y\in\sym^2(n)$ and $g\in\gl_n(K)$.

The space $\sym^2(n)$ is canonically identified
with the space of quadratic forms
in variables $v_1,\dots,v_n$;
the monomial $v_iv_j\in\sym^2(n)$
is the image of $e_i\otimes e_j\in n\otimes n$,
and $\gl_n(K)$ acts
by the linear change of the variables $v_1,\dots,v_n$.
The inverse image of
\[
x(v_1,\dots,v_n)=\sum_{i\leq j}x_{ij}v_iv_j\in\sym^2(n)
\]
in $\sym_2(n)$ via the isomorphism above is
\[
\sum_i x_{ii}\cdot e_i\otimes e_i
+\sum_{i<j} \frac{x_{ij}}2\cdot(e_i\otimes e_j+e_j\otimes e_i),
\]
so the bilinear form on $\sym^2(n)$ is given by
\[
[x,y]:=\sum_{i} x_{ii}y_{ii}+\frac12\sum_{i<j}x_{ij}y_{ij}
\]
for
\[
x=\sum_{i\leq j}x_{ij}v_iv_j,
\quad
y=\sum_{i\leq j}y_{ij}v_iv_j,
\]
and this satisfies $[gx,g^{-T}y]=[x,y]$ for $g\in\gl_n(K)$.

Similarly,
the composition
\[
\sym_3(2)
\hookrightarrow
2\otimes 2\otimes 2
\twoheadrightarrow
\sym^3(2)
\]
is an isomorphism if the characteristic of $K$ is not three.
The symmetric bilinear form on the space of binary cubic forms
$\sym^3(2)$ induced from $\sym_3(2)\subset 2\otimes 2\otimes 2$ is
\[
[x,y]=x_1y_1+\frac13x_2y_2+\frac13x_3y_3+x_4y_4
\]
for
\[
x(u,v)=x_1u^3+x_2u^2v+x_3uv^2+x_4v^3,
\quad
y(u,v)=y_1u^3+y_2u^2v+y_3uv^2+y_4v^3,
\]
and this satisfies $[gx,g^{-T}y]=[x,y]$ for $g\in\gl_2(K)$.

Thus we in particular have shown that the symmetric bilinear forms
we introduced in
\eqref{eq:sym32-bilinear-form}
\eqref{eq:sym22-bilinear-form},
\eqref{eq:sym23-bilinear-form},
\eqref{eq:2sym22-bilinear-form},
and
\eqref{eq:2sym23-bilinear-form}
respectively
for the spaces
$\sym^3(2)$, $\sym^2(2)$, $\sym^2(3)$, $2\otimes\sym^2(2)$
and $2\otimes\sym^2(3)$
all satisfy $[gx,g^{-T}y]=[x,y]$.

\bibliographystyle{alpha}
\bibliography{expo-sums}

\begin{thebibliography}{Hou18b}

\bibitem[BBP10]{BBP}
Karim Belabas, Manjul Bhargava, and Carl Pomerance.
\newblock Error estimates for the {D}avenport-{H}eilbronn theorems.
\newblock {\em Duke Math. J.}, 153(1):173--210, 2010.

\bibitem[BF99]{BF}
Karim Belabas and Etienne Fouvry.
\newblock Sur le 3-rang des corps quadratiques de discriminant premier ou
  presque premier.
\newblock {\em Duke Math. J.}, 98(2):217--268, 1999.

\bibitem[Bha04]{HCL3}
Manjul Bhargava.
\newblock Higher composition laws. {III}. {T}he parametrization of quartic
  rings.
\newblock {\em Ann. of Math. (2)}, 159(3):1329--1360, 2004.

\bibitem[Bha05]{B_quartic}
Manjul Bhargava.
\newblock The density of discriminants of quartic rings and fields.
\newblock {\em Ann. of Math. (2)}, 162(2):1031--1063, 2005.

\bibitem[Bha10]{B_quintic}
Manjul Bhargava.
\newblock The density of discriminants of quintic rings and fields.
\newblock {\em Ann. of Math. (2)}, 172(3):1559--1591, 2010.

\bibitem[BST13]{BST}
Manjul Bhargava, Arul Shankar, and Jacob Tsimerman.
\newblock On the {D}avenport-{H}eilbronn theorems and second order terms.
\newblock {\em Invent. Math.}, 193(2):439--499, 2013.

\bibitem[DG98]{DG}
Jan Denef and Akihiko Gyoja.
\newblock Character sums associated to prehomogeneous vector spaces.
\newblock {\em Compositio Math.}, 113(3):273--346, 1998.

\bibitem[DH71]{DH}
H.~Davenport and H.~Heilbronn.
\newblock On the density of discriminants of cubic fields. {II}.
\newblock {\em Proc. Roy. Soc. London Ser. A}, 322(1551):405--420, 1971.

\bibitem[DW86]{DW2}
Boris Datskovsky and David~J. Wright.
\newblock The adelic zeta function associated to the space of binary cubic
  forms. {II}. {L}ocal theory.
\newblock {\em J. Reine Angew. Math.}, 367:27--75, 1986.

\bibitem[Elk13]{elkies}
N.~Elkies.
\newblock {\em The rationality of conics over finite fields (And an
  introduction to rational normal curves and classical {G}oppa codes)}, 2013.
\newblock Lecture notes, available at
  \url{http://www.math.harvard.edu/~elkies/M256.13/conicfq.pdf}.

\bibitem[FK01]{FK}
E.~Fouvry and N.~Katz.
\newblock A general stratification theorem for exponential sums, and
  applications.
\newblock {\em J. Reine Angew. Math.}, 540:115--166, 2001.

\bibitem[Hou18a]{hough_quartic_expo}
Robert Hough.
\newblock The local zeta function in enumerating quartic fields.
\newblock {\em Preprint}, 2018.
\newblock Available at \url{https://arxiv.org/abs/1806.01372}.

\bibitem[Hou18b]{hough_quartic_shape}
Robert Hough.
\newblock The shape of quartic fields.
\newblock {\em Preprint}, 2018.
\newblock Avaiable at \url{https://arxiv.org/abs/1706.09987}.

\bibitem[Isha]{ishimoto-cubic}
K.~Ishimoto.
\newblock Orbital exponential sums for some cubic prehomogeneous vector spaces.
\newblock In preparation.

\bibitem[Ishb]{ishimoto-quadratic}
K.~Ishimoto.
\newblock Orbital exponential sums for some quadratic prehomogeneous vector
  spaces.
\newblock In preparation.

\bibitem[IT]{ishi-tani}
K.~Ishimoto and T.~Taniguchi.
\newblock Orbital exponential sums for classical type prehomogeneous vector
  spaces.
\newblock In preparation.

\bibitem[Mor10]{mori}
Shingo Mori.
\newblock Orbital {G}auss sums associated with the space of binary cubic forms
  over a finite field.
\newblock {\em RIMS K\^oky\^uroku}, 1715:32--36, 2010.

\bibitem[Nak98]{N}
Jin Nakagawa.
\newblock On the relations among the class numbers of binary cubic forms.
\newblock {\em Invent. Math.}, 134(1):101--138, 1998.

\bibitem[Ohn97]{O}
Yasuo Ohno.
\newblock A conjecture on coincidence among the zeta functions associated with
  the space of binary cubic forms.
\newblock {\em Amer. J. Math.}, 119(5):1083--1094, 1997.

\bibitem[PG14]{pari}
The PARI~Group.
\newblock {\em {PARI/GP version {\tt 2.8.0}}}.
\newblock Bordeaux, 2014.
\newblock Available from \url{http://pari.math.u-bordeaux.fr/}.

\bibitem[Sat90]{sato}
Mikio Sato.
\newblock Theory of prehomogeneous vector spaces (algebraic part)---the
  {E}nglish translation of {S}ato's lecture from {S}hintani's note.
\newblock {\em Nagoya Math. J.}, 120:1--34, 1990.
\newblock Notes by Takuro Shintani, Translated from the Japanese by Masakazu
  Muro.

\bibitem[Shi72]{shintani}
Takuro Shintani.
\newblock On {D}irichlet series whose coefficients are class numbers of
  integral binary cubic forms.
\newblock {\em J. Math. Soc. Japan}, 24:132--188, 1972.

\bibitem[SK77]{saki}
M.~Sato and T.~Kimura.
\newblock A classification of irreducible prehomogeneous vector spaces and
  their relative invariants.
\newblock {\em Nagoya Math. J.}, 65:1--155, 1977.

\bibitem[ST14]{ST}
Arul Shankar and Jacob Tsimerman.
\newblock Counting {$S_5$}-fields with a power saving error term.
\newblock {\em Forum Math. Sigma}, 2:e13, 8, 2014.

\bibitem[TT]{TT_leveldist}
Takashi Taniguchi and Frank Thorne.
\newblock Levels of distribution for sieve problems in prehomogeneous vector
  spaces.
\newblock Submitted (2017), preprint available at
  \url{https://arxiv.org/abs/1707.01850}.

\bibitem[TT13a]{TT_L}
Takashi Taniguchi and Frank Thorne.
\newblock Orbital {$L$}-functions for the space of binary cubic forms.
\newblock {\em Canad. J. Math.}, 65(6):1320--1383, 2013.

\bibitem[TT13b]{TT_rc}
Takashi Taniguchi and Frank Thorne.
\newblock Secondary terms in counting functions for cubic fields.
\newblock {\em Duke Math. J.}, 162(13):2451--2508, 2013.

\bibitem[Wri85]{Wright}
David~J. Wright.
\newblock The adelic zeta function associated to the space of binary cubic
  forms. {I}. {G}lobal theory.
\newblock {\em Math. Ann.}, 270(4):503--534, 1985.

\bibitem[WY92]{WY}
David~J. Wright and Akihiko Yukie.
\newblock Prehomogeneous vector spaces and field extensions.
\newblock {\em Invent. Math.}, 110(2):283--314, 1992.

\end{thebibliography}
 
 \newpage
 \newgeometry{top=0.6in,bottom=0.8in, left=0.8in}
 
\begin{theorem-non}\label{thm:explicit-big}
The matrix $q^{12} M$ of Theorem \ref{thm:2sym23-A} is as follows, with
$[abc]:=(q-1)^aq^b(q+1)^c$
and $\phi_2:=q^2+q+1$:
\begin{multline*}
\footnotesize
\left[
\begin{array}{ccccccccccc}
%	1
		1
&		[101]	\phi_2
&\tfrac12	[112]	\phi_2
&\tfrac12	[211]	\phi_2
&		[221]	\phi_2
&		[212]	\phi_2
&		[231]	\phi_2
&\tfrac12	[222]	\phi_2
&\tfrac12	[321]	\phi_2
&		[322]	\phi_2
\\
%	2
		1
&		d_1
&\tfrac12	[111]	c_8
&-\tfrac12	[111]
&		[120]	d_1
&		[111]	d_1
&		[232]
&-\tfrac12	[122]
&-\tfrac12	[221]
&-		[222]
\\
%	3
		1
&		[100]	c_8
&\tfrac12	[010]	e_3
&\tfrac12	[212]
&		[120]	d_1
&		[110]	d_{10}
&		[130]	c_3
&\tfrac12	[220]	c_{12}
&\tfrac12	[220]	c_1
&		[221]	c_1
\\
%	4
		1
&-		[001]
&\tfrac12	[113]
&\tfrac12	[010]	e_1
&		[120]	d_1
&-		[112]
&		[231]
&-\tfrac12	[121]	\phi_2
&\tfrac12	[121]	c_2
&-		[221]	\phi_2
\\
%	5
		1
&			d_1
&\tfrac12	[011]	d_1
&\tfrac12	[110]	d_1
&		[020]	e_2
&-		[112]
&-		[131]
&-\tfrac12	[122]
&-\tfrac12	[221]
&		[121]
\\
%	6
		1
&			d_1
&\tfrac12	[010]	d_{10}
&-\tfrac12	[111]	
&-		[121]	
&		[010]	e_4
&		[130]	c_1
&\tfrac12	[120]	d_5
&\tfrac12	[120]	d_2
&-		[120]	c_1
\\
	%7
		1
&		[102]
&\tfrac12	[011]	c_3
&\tfrac12	[211]	
&-		[121]	
&		[111]	c_1
&-		[030]	b_{-2}
&\tfrac12	[222]	
&\tfrac12	[321]	
&-		[221]	
\\
	%8
		1
&-		[001]
&\tfrac12	[110]	c_{12}
&-\tfrac12	[110]	\phi_2
&-		[121]	
&		[110]	d_5
&		[231]	
&-\tfrac12	[020]	d_3
&-\tfrac12	[221]	
&		[121]	
\\
	%9
		1
&-		[001]
&\tfrac12	[011]	c_1
&\tfrac12	[011]	c_2
&-		[121]	
&		[011]	d_2
&		[231]	
&-\tfrac12	[122]	
&-\tfrac12	[020]	d_4
&		[121]	
\\
	%10
		1
&-		[001]
&\tfrac12	[011]	c_1
&-\tfrac12	[110]	\phi_2
&		[020]	
&-		[010]	c_1
&-		[130]	
&\tfrac12	[021]	
&\tfrac12	[120]	
&-		[020]	
\\
	%11
		1
&		c_1
&-\tfrac12	[010]	a_7
&-\tfrac12	[110]	
&		[020]	
&		[010]	d_7
&-		[030]	a_1
&-\tfrac12	[121]	
&-\tfrac12	[220]	
&		[120]	b_1
\\
	%12
		1
&			c_1
&\tfrac12	[010]	d_6
&-\tfrac12	[110]	b_1
&-		[120]	
&		[110]	c_3
&-2		[130]	
&-\tfrac12	[121]	
&-\tfrac12	[220]	
&-		[220]	
\\
	%13
		1
&-			\phi_2
&\tfrac12	[112]	
&-\tfrac12	[010]	d_3
&-		[120]	
&-		[112]	
&		0	
&\tfrac12	[022]	
&\tfrac12	[121]	
&		[121]	
\\
	%14
		1
&		[100]	a_4
&\tfrac12	[010]	c_6
&\tfrac12	[210]	
&-		[120]	
&		[010]	d_9
&-		[030]	a_2
&-\tfrac12	[120]	
&-\tfrac12	[120]	a_3
&		[320]	
\\
	%15
		1
&-		[001]
&\tfrac12	[010]	c_3
&\tfrac12	[010]	b_1
&-		[120]	
&		[010]	d_7
&-		[130]	
&-\tfrac12	[120]	a_4
&\tfrac12	[120]	
&-		[221]	
\\
	%16
		1
&			c_{10}
&\tfrac12	[010]	c_9
&\tfrac12	[110]	a_3
&-		[020]	a_3
&-		[010]	a_3a_7
&4		[030]	
&\tfrac12	[020]	c_{11}
&-\tfrac12	[120]	a_6
&-		[020]	a_3a_6
\\
	%17
		1
&			c_1
&-\tfrac12	[010]	a_7
&-\tfrac12	[110]	
&		[020]	
&-		[110]	a_4
&2		[030]	
&-\tfrac12	[121]	
&\tfrac12	[020]	c_4
&		[120]	
\\
	%18
		1
&-			\phi_2
&\tfrac12	[110]	a_4
&\tfrac12	[010]	c_7
&-		[020]	a_3
&-		[110]	a_4
&		0
&-\tfrac12	[020]	c_3
&\tfrac12	[121]	
&		[021]	a_3
\\
	%19
		1
&-		[001]
&-\tfrac12	[012]	
&-\tfrac12	[111]	
&		[021]	
&		[011]	
&		[030]	
&\tfrac12	[021]	
&\tfrac12	[120]	
&-		[021]	
\\
	%20
		1
&-			\phi_2
&-\tfrac12	[011]
&\tfrac12	[011]	
&		[020]	
&		[011]	
&		0	
&\tfrac12	[022]	
&-\tfrac12	[020]	b_1
&-		[021]	
\\
\end{array}
\right.\\
\footnotesize
\left.
\begin{array}{cccccccccc}
	%1
		[332]	\phi_2
&\tfrac12	[242]	\phi_2
&\tfrac12	[341]	\phi_2
&\tfrac12	[342]	\phi_2
&\tfrac12	[342]	\phi_2
&\tfrac1{24}	[442]	\phi_2
&\tfrac14	[442]	\phi_2
&\tfrac18	[442]	\phi_2
&\tfrac13	[442]	\phi_2
&\tfrac14	[442]	\phi_2
\\
	%2
\		[231]	c_1\
&\tfrac12	[141]	c_1
& -\tfrac12	[240]	\phi_2
&\tfrac12	[341]	a_4
&-\tfrac12	[242]	
&\tfrac1{24}	[341]	c_{10}
&\tfrac14	[341]	c_1
&-\tfrac18	[341]	\phi_2
&-\tfrac13	[342]	
&-\tfrac14	[341]	\phi_2
\\
	%3
-		[230]	a_7
&\tfrac12	[140]	d_6
&\tfrac12	[341]	
&\tfrac12	[240]	c_6
&\tfrac12	[240]	c_3
&\tfrac1{24}	[340]	c_9
&-\tfrac14	[340]	a_7
&\tfrac18	[440]	a_4
&-\tfrac13	[342]	
&-\tfrac14	[341]	
\\
	%4
-		[231]
&-\tfrac12	[141]	b_1
&-\tfrac12	[140]	d_3
&\tfrac12	[341]	
&\tfrac12	[141]	b_1
&\tfrac1{24}	[341]	a_3
&-\tfrac14	[341]	
&\tfrac18	[241]	c_7
&-\tfrac13	[342]	
&\tfrac14	[242]	
\\
	%5
\		[131]\
&-\tfrac12	[141]	
&-\tfrac12	[240]	
&-\tfrac12	[241]	
&-\tfrac12	[241]	
&-\tfrac1{24}	[241]	a_3
&\tfrac14	[241]	
&-\tfrac18	[241]	a_3
&\tfrac13	[242]	
&\tfrac14	[241]	
\\
	%6
\		[130]	d_7\
&\tfrac12	[140]	c_3
&-\tfrac12	[241]	
&\tfrac12	[140]	d_9
&\tfrac12	[140]	d_7
&-\tfrac1{24}	[240]	a_3a_7
&-\tfrac14	[340]	a_4
&-\tfrac18	[340]	a_4
&\tfrac13	[241]	
&\tfrac14	[241]	
\\
	%7
-		[131]	a_1
&-		[141]	
&0
&-\tfrac12	[141]	a_2
&-\tfrac12	[241]	
&\tfrac1{6}	[241]	
&\tfrac12	[241]	
&0
&\tfrac13	[241]	
&0
\\
	%8
-		[231]
&-\tfrac12	[141]	
&\tfrac12	[141]	
&-\tfrac12	[240]	
&-\tfrac12	[240]	a_4
&\tfrac1{24}	[240]	c_{11}
&-\tfrac14	[341]	
&-\tfrac18	[240]	c_3
&\tfrac13	[241]	
&\tfrac14	[242]	
\\
	%9
-		[231]
&-\tfrac12	[141]	
&\tfrac12	[141]	
&-\tfrac12	[141]	a_3
&\tfrac12	[141]	
&-\tfrac1{24}	[241]	a_6
&\tfrac14	[141]	c_4
&\tfrac18	[242]	
&\tfrac13	[241]	
&-\tfrac14	[141]	b_1
\\
	%10
\		[130]	b_1\
&-\tfrac12	[140]	
&\tfrac12	[140]	
&\tfrac12	[340]	
&-\tfrac12	[241]	
&-\tfrac1{24}	[140]	a_3a_6
&\tfrac14	[240]	
&\tfrac18	[141]	a_3
&-\tfrac13	[141]	
&-\tfrac14	[141]	
\\
	%11
\		[030]	d_8
&-		[140]	
&0
&-\tfrac12	[140]	a_5
&\tfrac12	[140]	
&\tfrac1{6}	[140]	a_3
&-\tfrac12	[140]	
&0
&-\tfrac13	[141]	
&0
\\
	%12
-2		[230]
&\tfrac12	[040]	c_5
&-\tfrac12	[240]	
&-		[140]	a_1
&		[140]	
&\tfrac1{4}	[240]	
&\tfrac12	[240]	
&\tfrac14	[240]	
&0
&0
\\
	%13
0
&-\tfrac12	[141]	
&\tfrac12	[040]	b_{-3}
&0
&		[141]	
&0
&0
&-\tfrac12	[141]	
&0
&-\tfrac12	[141]	
\\
	%14
-		[130]	a_5
&-		[040]	a_1
&0
&\tfrac12	[040]	a_8
&\tfrac12	[140]	
&-\tfrac1{2}	[140]	
&-\tfrac12	[140]	
&0
&0
&0
\\
	%15
\		[130]	\
&		[040]	
&		[140]	
&\tfrac12	[140]	
&\tfrac12	[040]	a_2
&0
&-\tfrac12	[140]	
&-\tfrac12	[140]	
&0
&0
\\
	%16
4		[030]	a_3
&3		[040]	
&0
&-6		[040]	
&0
&\tfrac1{24}	[040]	b_{23}
&-\tfrac14	[141]	
&\tfrac18	[141]	
&\tfrac13	[141]	
&-\tfrac14	[141]	
\\
	%17
-2		[030]
&		[040]	
&0
&-		[040]	
&-		[040]	
&-\tfrac1{24}	[141]	
&\tfrac14	[040]	b_3
&-\tfrac18	[141]	
&-\tfrac13	[141]	
&\tfrac14	[141]	
\\
	%18
0
&		[040]	
&-2		[040]	
&0
&-2		[040]	
&\tfrac1{24}	[141]	
&-\tfrac14	[141]	
&\tfrac18	[040]	b_7
&\tfrac13	[141]	
&-\tfrac14	[141]	
\\
	%19
-		[031]
&0
&0
&0
&0
&\tfrac1{24}	[141]	
&-\tfrac14	[141]	
&\tfrac18	[141]	
&\tfrac13	[040]	b_2
&-\tfrac14	[141]	
\\
	%20
0
&0
&-		[040]	
&0
&0
&-\tfrac1{24}	[141]	
&\tfrac14	[141]	
&-\tfrac18	[141]	
&-\tfrac13	[141]	
&\tfrac14	[040]	b_3
\\
\end{array}
\right].
\end{multline*}
Here $a_i$, $b_i$, $c_i$, $d_i$ and $e_i$
are irreducible integral polynomials in $q$, 
defined by $b_i=q^2+i$ and

\begin{minipage}[t]{30mm}
\begin{align*}
a_1&=	q-2,\\
a_2&=	q-3,\\
a_3&=	2q-1,\\
a_4&=	2q+1,\\
a_5&=	2q-3,\\
a_6&=	3q-1,\\
a_7&=	3q+1,\\
a_8&=	5q-7,
\end{align*}
\end{minipage}
\begin{minipage}[t]{40mm}
\begin{align*}
c_1&=	q^2-q-1,\\
c_2&=	q^2-q+1,\\
c_3&=q^2-2q-1,\\
c_4&=q^2+2q-1,\\
c_5&=q^2-2q+3,\\
c_6&=q^2-4q-1,\\
c_7&=2q^2+q+1,\\
c_8&=2q^2+2q+1,\\
c_9&=2q^2-5q-1,\\
c_{10}&=3q^2-q-1,\\
c_{11}&=3q^2-2q+1,\\
c_{12}&=3q^2+3q+1,
\end{align*}
\end{minipage}
\begin{minipage}[t]{45mm}
\begin{align*}
d_1&=	q^3-q-1,\\
d_2&=	q^3-q^2+1,\\
d_3&=	q^3-q^2-q-1,\\
d_4&=	q^3+q^2-q+1,\\
d_5&=	q^3-q^2-2q-1,\\
d_6&=	q^3+q^2-3q-1,\\
d_7&=	q^3-2q^2+q+1,\\
d_8&=	q^3-2q^2+2q-2,\\
d_9&=	q^3-4q^2+q+1,\\
d_{10}&=2q^3-q^2-2q-1,
\end{align*}
\end{minipage}
\begin{minipage}[t]{45mm}
\begin{align*}
e_1&=	q^4+1,\\
e_2&=	q^4-q^3+1,\\
e_3&=	q^4+2q^3-2q^2-2q-1,\\
e_4&=	2q^4-2q^3-q^2+q+1.
\end{align*}
\end{minipage}
\end{theorem-non}

\end{document}